\documentclass[12pt, reqno]{amsart}
\usepackage{fullpage,url,amssymb,amsmath,amsthm,amsfonts,mathrsfs}
\usepackage[usenames,dvipsnames]{color}
\usepackage[pagebackref = true, colorlinks = true, linkcolor = blue, citecolor = Green]{hyperref}
\usepackage[alphabetic,lite]{amsrefs}
\usepackage{enumitem}
\usepackage{amscd}   % for commutative diagrams
\usepackage[all, cmtip]{xy} % for complicated commutative diagrams
\usepackage{xfrac}

\DeclareFontEncoding{OT2}{}{} % to enable usage of cyrillic fonts
\newcommand{\textcyr}[1]{%
 {\fontencoding{OT2}\fontfamily{wncyr}\fontseries{m}\fontshape{n}\selectfont #1}}
\newcommand{\Sha}{{\mbox{\textcyr{Sh}}}}
\newcommand{\sha}{{\mbox{\textcyr{sh}}}}

\newcommand{\ssha}{\mbox{\tiny\textcyr{sh}}}

\usepackage{color} 
\newtheorem{lemma}{Lemma}[section]
\newtheorem{theorem}[lemma]{Theorem}
\newtheorem{prop}[lemma]{Proposition}
\newtheorem{cor}[lemma]{Corollary}

\newtheorem{example}[lemma]{Example}

\theoremstyle{definition}
\newtheorem{remark}[lemma]{Remark}

\newcommand{\A}{{\mathbb A}}
\newcommand{\G}{{\mathbb G}}

\newcommand{\PP}{{\mathbb P}}

\newcommand{\F}{{\mathbb F}}
\newcommand{\Q}{{\mathbb Q}}
\newcommand{\R}{{\mathbb R}}
\newcommand{\Z}{{\mathbb Z}}

\newcommand{\Kbar}{{\overline{K}}}
\newcommand{\Lbar}{{\overline{L}}}

\newcommand{\Xbar}{{\overline{X}}}
\newcommand{\Cbar}{{\overline{C}}}

\newcommand{\kk}{{\mathbf k}}

\newcommand{\mm}{{\mathfrak m}}
\newcommand{\mathfrakL}{\mathfrak L}

\newcommand{\calA}{{\mathcal A}}

\newcommand{\calG}{{\mathcal G}}

\DeclareMathOperator{\inv}{inv}
\DeclareMathOperator{\eval}{eval}

\DeclareMathOperator{\im}{im}

\DeclareMathOperator{\Gal}{Gal}

\DeclareMathOperator{\Cor}{Cor}

\DeclareMathOperator{\Norm}{Norm}
\DeclareMathOperator{\Map}{Map}
\DeclareMathOperator{\Br}{Br}

\DeclareMathOperator{\Sel}{Sel}

\DeclareMathOperator{\divv}{div}

\DeclareMathOperator{\Div}{Div}
\DeclareMathOperator{\Alb}{Alb}
\DeclareMathOperator{\Pic}{Pic}
\DeclareMathOperator{\NS}{NS}
\DeclareMathOperator{\Jac}{Jac}

\DeclareMathOperator{\Spec}{Spec}

\DeclareMathOperator{\HH}{H}
\DeclareMathOperator{\N}{N}

\DeclareMathOperator{\res}{res}
\DeclareMathOperator{\Princ}{Princ}

\newcommand{\To}{\longrightarrow}

\numberwithin{equation}{section}
\numberwithin{table}{section}

\newcommand{\defi}[1]{\textsf{#1}} % for defined terms

\title{Two torsion in the Brauer group of a hyperelliptic curve}

\author{Brendan Creutz}
\address{Department of Mathematics and Statistics, University of Canterbury, Private Bag 4800, Christchurch 8140, New Zealand}
\email{brendan.creutz@canterbury.ac.nz}
\urladdr{http://www.math.canterbury.ac.nz/\~{}b.creutz}

\author{Bianca Viray}
\thanks{The second author was partially supported by NSF grant DMS-1002933.}
\address{Department of Mathematics, Box 1917, Brown University, Providence, RI 02912, USA}
\email{bviray@math.brown.edu}
\urladdr{http://math.brown.edu/\~{}bviray}

\date{}

\begin{document}
	%%%%%%%%%%%%%%%%%%%%%%%%%%%%%%%%%%%%%%%%%%%%%%%%%%%%%%%%%%%%%%%%%%%%%%%%%%%%
	\begin{abstract}
		We construct unramified central simple algebras representing $2$-torsion classes in the Brauer group of a hyperelliptic curve, and show that every $2$-torsion class can be constructed this way when the curve has a rational Weierstrass point or when the base field is $C_1$. In general, we show that a large (but in general proper) subgroup of the $2$-torsion classes are given by the construction. Examples demonstrating applications to the arithmetic of hyperelliptic curves defined over number fields are given.
		% plain text abstract
		% We construct unramified central simple algebras representing 2-torsion classes in the Brauer group of a hyperelliptic curve, and show that every 2-torsion class can be constructed this way when the curve has a rational Weierstrass point or when the base field is C_1. In general, we show that a large (but in general proper) subgroup of the 2-torsion classes are given by the construction. Examples demonstrating applications to the arithmetic of hyperelliptic curves defined over number fields are given.
	\end{abstract}
	%%%%%%%%%%%%%%%%%%%%%%%%%%%%%%%%%%%%%%%%%%%%%%%%%%%%%%%%%%%%%%%%%%%%%%%%%%%%
		
	\maketitle
	% \tableofcontents
	% \pagebreak

	%%%%%%%%%%%%%%%%%%%%%%%%%%%%%
	\section{Introduction}
	%%%%%%%%%%%%%%%%%%%%%%%%%%%%%

		Let $X$ be a smooth, projective and geometrically integral variety over a field $K$ of characteristic different from $2$. The \defi{Brauer group} of $X$, denoted $\Br X$, is a generalization of the usual notion of the Brauer group of a field. By the Purity Theorem~\cite{Fujiwara-purity}, $\Br X$ is the unramified subgroup of $\Br \kk(X)$, the Brauer group of the function field of $X$, and it is known~\cite{Merkurjev} that any $2$-torsion element of $\Br X$ can be written as a tensor product of quaternion algebras over $\kk(X)$. The goal of this paper is to make this explicit in the case that $X=C$ is a double cover of the projective line, henceforth denoted by $\pi : C \to \PP^1$. To our knowledge, this has only been achieved in special cases where $\pi$ is assumed to have $K$-rational branch points and the Jacobian of $C$ has rational $2$-torsion (see~\cite{Wittenberg-transcendental}*{Prop. 2.2}, \cite{Skorobogatov-torsors}*{p.91}, and~\cite{RTY}).
		
		We show how unramified central simple algebras over $\kk(C)$ can be constructed  from functions defined on the branch locus of $\pi$, and determine precisely when two such functions yield the same Brauer class. Furthermore, we show that every $2$-torsion Brauer class can be constructed this way when $\pi$ has a $K$-rational branch point or when $K$ is a $C_1$ field. The former case includes all elliptic curves. The latter applies when $C$ is the generic fiber of a double cover of a ruled surface over a separably closed field. This is used in \cite{CreutzViray} to obtain an equally explicit presentation of the $2$-torsion in the Brauer group of such a surface.
		
		Of course, much of the interest in Brauer groups of varieties arises from arithmetic applications in which the variety is not known to have rational points and the base field is emphatically not $C_1$. Specifically, if $K$ is a global field, then, as Manin~\cite{Manin-BMobs} observed, elements of the Brauer group can obstruct the existence of $K$-points, even when there is no local obstruction. Over global fields our construction yields a presentation of a subgroup of $(\Br C)[2]$ which can (and usually will) be proper, but remains large enough for interesting applications to the study of rational points.
		
		To illustrate, let us consider the example $C:y^2 = 2(x^4-17)$. It has been shown \cite{Lind,Reichardt} that $C$ is locally solvable, yet has no rational points. By work of Cassels~\cite{CasselsIV} and Manin ~\cite{Manin-BMobs} it is known that this failure of the local-global principle must be attributable to some element in $\Br C$. Our results allow us to write down such an element explicitly. Namely, the quaternion algebra $\calA := (-x^2-4,-2)_2$ represents an element of $\Br C$ which obstructs the existence of rational points on $C$. Indeed, a rational point $P = (x_0,y_0) \in C(\Q)$ would give rise to a quaternion algebra, $\eval_P(\calA) = (-x_0^2-4,-2)_2 \in \Br\Q$, which would necessarily be trivial over $\Q_p$ for every finite prime $p$, but  nontrivial over $\R$.\footnote{For an odd prime $p$ the algebra can only be nontrivial if $x_0^4 \equiv 16 \bmod p$ (as $-x_0^2-4$ must have odd valuation), but for this one must have $-2 \in \Q_p^{\times 2}$ since $-2 \equiv 2(x_0^4-17) \bmod p$. Similarly, to satisfy the equation defining $C$ one must have $x_0 \equiv \pm 3 \bmod 8$, from which it follows that $(-x_0^4-4,-2)_2$ is isomorphic to the trivial algebra $(3,-2)_2$ over $\Q_2$. The algebra is nontrivial over $\R$ because $-x_0^2-4$ is negative.} Thus, the class $\calA \in \Br C$ allows us to see that the existence of a rational point on $C$ is incompatible with the reciprocity law in the Brauer group of $\Q$.
		
		In general, if $C$ is defined over a number field and is locally solvable, our method allows us to construct explicit representatives for all $2$-torsion Brauer classes that are locally constant as in the example above (see Theorem~\ref{thm:LocallyConstantsInImage} for the precise statement). Modulo constant algebras, these classes correspond to elements of $\Sha(J)$, the Shafarevich-Tate group of $J$. As shown in \cite[Th\'eor\`eme 6]{Manin-BMobs}, the obstruction given by such an algebra is closely related to the Cassels-Tate pairing on $\Sha(J)$. Using this, our results give a practical algorithm for computing the Cassels-Tate pairing between the elements of $\Sha(J)[2]$ and the torsor $\Pic^1_C \in \Sha(J)$ parameterizing divisor classes of degree $1$ on $C$. In an example we carry out such computations for a family of quadratic twists, giving an infinite family of abelian surfaces over $\Q$ with nontrivial Shafarevich-Tate group.

	%%%%%%%%%%%%%%%%%%%%%%%%%%%%%%		
	\subsection{Construction of the algebras}
	%%%%%%%%%%%%%%%%%%%%%%%%%%%%%%
		Let $K$ be a field of characteristic different from $2$ and let $\pi : C \to \PP^1$ be an {irreducible} double cover of the projective line defined over $K$ with Jacobian $J := \Jac(C)$. We say that $\pi$ is \defi{odd} if $\pi$ is ramified above  $\infty \in \PP^1(K)$. Otherwise we say that $\pi$ is \defi{even}. Provided $K$ has sufficiently many elements (e.g. if $K$ is infinite) a change of coordinates on $\PP^1$ allows us to obtain an isomorphic double cover which is even. On the other hand, $\pi$ is isomorphic to an odd double cover if and only there is a $K$-rational ramification point. While there is thus no loss of generality in considering only even double covers, it is possible to obtain results that are sharper in the case of odd double covers (cf. Theorem \ref{thm:OddHypThm} and Remark \ref{rem:NeedHypInf}). We have chosen the notation below to allow the two cases to be treated in parallel.
		
	By Kummer theory, $C$ has a model of the form $y^2 = cf(x)$ with $c \in K^\times$ and $f(x)$ a square free monic polynomial with coefficients in $K$. Moreover, when $C$ (or equivalently $\deg(f)$) is odd, we can perform a change of coordinates to arrange that $c = 1$. Let $\Omega \subseteq C$ be the set of ramification points of $\pi$, and let $L = \Map_K(\Omega,\Kbar)$ denote the \'etale $K$-algebra corresponding to $\Omega$. When $C$ is even we may identify $K[\theta]/f(\theta)$ with $L$. When $C$ is odd, $K[\theta]/f(\theta)$ can be identified with the subalgebra $L_\circ \subseteq L$ consisting of elements $\ell \in L = \Map_K(\Omega,\Kbar)$ that take the value $1$ at the ramification point above $\infty \in \PP^1(K)$. In the odd case this gives a canonical isomorphism $L \cong L_\circ \times K$.
	
	Let $x - \alpha$ denote the image of $x - \theta$ in $\kk(C_L) := L \otimes_K \kk(C)$; in the odd case this means $x-\alpha$ is the image of $(x-\theta,1)$ in $\kk(C_{L_\circ})\times\kk(C)$. In~\cites{Schaefer-descent,PS-descent} $x-\alpha$ is applied to the classical problem of descents on Jacobians of hyperelliptic curves (see \S\ref{subsec:Outline}). In this paper we study the related homomorphism,
	\[
		\gamma' \colon L^\times \to (\Br \kk(C))[2], \quad \ell \mapsto \Cor_{\kk(C_L)/\kk(C)}((\ell,x-\alpha)_2)\,.
	\]
	This map constructs a central simple algebra over $\kk(C)$ from an element $\ell \in L^\times$. Proposition~\ref{prop:RT} below shows how to write $\gamma'(\ell)$ as a tensor product of quaternion algebras over $\kk(C)$. For example, when $C : y^2 = 2(x^4-17)$ the element $\ell = (-\theta^2-4) \in L^\times$ is mapped by $\gamma'$ to the quaternion algebra $\calA = (-x^2-4,-2)_2$ considered above (see Example~\ref{ex:RLcurve}).
	
	%%%%%%%%%%%%%%%%%%%%%%%%%%%%%%		
	\subsection{Statement of the results}
	%%%%%%%%%%%%%%%%%%%%%%%%%%%%%%
	We provide answers to the following questions:
	
	\begin{enumerate}
		\item Which elements of $L^\times$ are mapped by $\gamma'$ into $\Br C$?
		\item When do two elements of $L^\times$ map to the same class in $\Br \kk(C)$?
		\item Which elements of $\Br C$ lie in the image of $\gamma'$ and, in particular, when does $\gamma'$ surject onto $\Br C[2]?$
	\end{enumerate}
	
	Set %$\mathfrakL = \frac{L^\times}{K^\times L^{\times 2}}$.
	\[
		\mathfrakL = \frac{L^\times}{K^\times L^{\times 2}}
	\]
	For $a \in K^\times$ and $\ell \in L^\times$, we use $\overline{a}$ and $\overline\ell$ to denote the corresponding classes in $K^\times/K^{\times2}$ and $\mathfrakL$, and set
	\[  
		\mathfrak{L}_a = \left\{ \overline{\ell} \in \mathfrakL \,:\, \Norm_{L/K}\left(\overline{\ell}\right)\in\langle \overline{a}\rangle \right\}\,, 
	\]
	where $\Norm_{L/K}$ denotes the map $\mathfrakL \to K^\times/K^{\times 2}$ induced by the norm on $L$. Note that when $C$ is odd we have a canonical isomorphism $\mathfrakL \cong L^\times_\circ/L^{\times 2}_\circ$ under which $\Norm_{L/K}$ coincides with the map induced by the norm on $L_\circ$.
	
	The first question above is answered by the following.
	\begin{theorem}
		\label{thm:gammawelldefined}
		Let $\ell \in L^\times$. If $C$ is odd, then $\gamma'(\ell) \in \Br C$. If $C$ is even, then $\gamma'(\ell) \in \Br C$ if and only if $\overline{\ell} \in \mathfrakL_c$.		
	\end{theorem}

	As shown in \cites{Schaefer-descent,PS-descent} $x - \alpha$ induces a homorphism $\Pic C \to\, \mathfrak{L}$, where $\Pic C$ is the Picard group of $C$. This is defined as follows. For a closed point $P \in C \setminus (\Omega \cup \pi^{-1}(\infty))$ one defines $x(P)-\alpha = \prod_{i = 1}^d (x_i - \alpha) \in L^\times$, where $P(\Kbar) = \{(x_1,y_1), \ldots, (x_d,y_d)\}$. Every divisor class $[D] \in \Pic C$ can be represented by a sum $\sum_{P} n_PP$ of such closed points, and $(x-\alpha)([D])$ is defined to be the class of $\prod_P (x(P)-\alpha)^{n_P}$ in $\mathfrakL$.
		
	\begin{theorem}
		\label{thm:MainComplex}
		The map $\gamma'$ induces a complex of abelian groups
		\begin{equation*}
			\label{eq:Complex}
			\mathfrak{C}: \quad \frac{\Pic C}{2\Pic C} \stackrel{x-\alpha}\To \mathfrak{L}_c \stackrel{\gamma}\To \left(\frac{\Br C}{\Br_0 C}\right)[2] \stackrel{\partial_0}\To 0\,,
		\end{equation*}
		where $\Br_0C := \im\left(\Br K \to \Br C\right)$ is the subgroup of constant algebras.
	\end{theorem}
	
	\begin{remark}
		If $C$ and $\pi$ are defined over a subfield $K_0 \subset K$ such that $K/K_0$ is Galois, then all of the groups in $\frak{C}$ have a natural action of $\Gal(K/K_0)$ and one sees from the definitions that the maps are $\Gal(K/K_0)$-equivariant.
	\end{remark}
	
	The second and third questions above ask to compute the groups
	\[
				\HH_1(\mathfrak{C}) := \frac{\ker(\gamma)}{\im(x-\alpha)}\,\text{ and }
				\HH_0(\mathfrak{C}) := \frac{\ker(\partial_0)}{\im(\gamma)}\,.
	\]
	For good measure, the group $\HH_2(\mathfrak{C}):= \ker(x-\alpha)$ is determined in Proposition~\ref{prop:dimImage}.
	\begin{theorem}
		\label{thm:ComputeH1}
		Let $\Pic^1C$ denote the set of divisor classes of degree $1$ on $C$ and let $\Pic^1_C(K)$ denote the subset of divisor classes of degree $1$ on $C_\Kbar$ that are fixed by Galois. Then
		\[
			\HH_1(\mathfrak{C}) = 
			\begin{cases}
				\mathfrak{L}_c/\mathfrak{L}_1 & \text{if $\Pic^1C = \emptyset \ne \Pic^1_C(K)$,} \\
				0 & \text{ otherwise.}
			\end{cases}
		\]
		In particular, this group has order at most $2$.
	\end{theorem}
	
	Our description of the group $\HH_0(\mathfrak{C})$ is more difficult to state (see Proposition~\ref{prop:L1exact}). The next three theorems give special cases in which it is trivial.
		
	\begin{theorem}
		\label{thm:ExactIfC1}
		If $(\Br K)[2] = 0$, then $\mathfrak{C}$ is an exact sequence.
	\end{theorem}
	
	\begin{theorem}
		\label{thm:EvenOddThm}
		If $\Omega$ admits a Galois-stable partition into two sets of odd cardinality, then $\mathfrak{C}$ has an exact subcomplex
		\[
			0 \to J(K)/2J(K) \stackrel{x-\alpha}\To \mathfrak{L}_1 \stackrel{\gamma}\To \left(\frac{\Br C}{\Br_0 C}\right)[2] \to 0\,.
		\]
	\end{theorem}
		
	\begin{theorem}
		\label{thm:OddHypThm}
		If $C$ is odd, then $\gamma'$ induces an exact sequence
		\[
			0 \to J(K)/2J(K) \stackrel{x-\alpha}\To \mathfrak{L}_1 \stackrel{\gamma}\To (\Br^0C)[2] \to 0\,,
		\]
		where $\Br^0C$ denotes the subgroup of $\Br C$ consisting of Brauer classes that evaluate to $0$ at the $K$-rational ramificiation point of $C$ lying above $\infty \in \PP^1(K)$. 
	\end{theorem}	
	
	When $K$ is a number field and $C$ is generic, $\HH_0(\mathfrak{C})$ does not vanish (see Remark~\ref{rem:Br2neBrUps}). However, as mentioned above, the image of $\gamma$ is large enough for interesting arithmetic applications.
										
	\begin{theorem}
		\label{thm:LocallyConstantsInImage}
		Suppose that $K$ is a number field and $C$ is locally solvable. If $\calA \in (\Br C)[2]$ is locally constant (in the sense that $(\calA \otimes K_v) \in \Br_0 C_{K_v}$ for each completion $K_v$ of $K$), then the image of $\calA$ in $\HH_0(\mathfrak{C})$ is trivial.
	\end{theorem}
	
	Some applications of this result are explored in \S\ref{sec:CTpairing}.

%%%%%%%%%%%%%%%%%%%%%%%%%%%%
	\subsection{Outline of the proofs}
	\label{subsec:Outline}
%%%%%%%%%%%%%%%%%%%%%%%%%%%%

	Theorems~\ref{thm:gammawelldefined} and~\ref{thm:MainComplex} are proved by explicitly computing the residues of an algebra of the form $\Cor_{\kk(C_L)/\kk(C)}((\ell,x-\alpha)_2)$, and then applying the purity theorem~\cite{Fujiwara-purity}. This is carried out in section \S\ref{sec:gamma}.

	The proofs of the other theorems are inspired by the classical problem of $2$-descents on Jacobians of hyperelliptic curves. Here one attempts to compute $J(K)/2J(K)$ by describing its image under the connecting homomorphism $\delta$ in the Kummer sequence,
	\begin{equation}\label{eq:Kummer}
		0 \to J(K)/2J(K) \stackrel{\delta}\to \HH^1(K,J[2]) \to \HH^1(K,J)[2] \to 0\,.
	\end{equation}
	To make use of this in practice, one requires concrete descriptions of $\HH^1(K,J[2])$ and the map $\delta$. When $C$ is odd this is achieved in \cite{Schaefer-descent} by giving an explicit isomorphism $\HH^1(K,J[2]) \simeq \mathfrakL_1$ whose composition with $\delta$ is equal to the $x-\alpha$ map.  Moreover, the existence of a rational point implies an isomorphism $\HH^1(K,J)[2] \simeq \Br^0(C)[2]$. Together with~\eqref{eq:Kummer} these isomorphisms imply the existence of an exact sequence as stated in Theorem \ref{thm:OddHypThm}, and the task is to verify that the description of $\gamma$ given is correct. This will ultimately be achieved by a cocycle computation. In the case that $C = J$ is an elliptic curve with rational $2$-torsion this has been carried out in \cite{Wittenberg-transcendental}*{Prop. 2.2} (see also \cite{Skorobogatov-torsors}*{p.91}).
		
	When $C$ is even, there are complications due to the fact that, in general, neither of the aforementioned isomorphisms exist. In the first instance we are forced to replace the isomorphism of \cite{Schaefer-descent} with the {\em fake descent} setup of \cite{PS-descent}. This implies the existence of an exact sequence,
	\begin{equation}\label{eq:fakedescent}
		\frac{\Pic^0C}{2\Pic^0C} \stackrel{x-\alpha}\To \mathfrakL_1 \stackrel{d}{\To} \frac{\HH^1(K,J)[2]}{\langle \Pic^1_C\rangle}\,,
	\end{equation}
	where under suitable hypothesis (e.g. if $\Br K[2] = 0$) the final map is surjective.	
	When $C$ has no $K$-rational divisors of degree $1$ the second isomorphism above must be replaced by an exact sequence,
	\begin{equation}
		\label{seq:h}
		0 \to  \left(\frac{\Br C}{\Br_0 C}\right)[2]\, \stackrel{{h_0}}\to \left(\frac{\HH^1(K,J)}{\langle \Pic^1_C \rangle}\right)[2] \to \HH^3(K,\Kbar^\times)\,.
	\end{equation}
	This means that, even under the assumption that $K$ is $C_1$, the image of the map $d$ in~\eqref{eq:fakedescent} may only correspond to an index $2$ subgroup of $(\Br C/\Br_0 C)[2]$. Our solution to this problem is inspired by \cite{CreutzANTSX} where it is shown how the elements of $\mathfrakL_c \setminus \mathfrakL_1$ correspond to certain $\Pic^1_C$-torsors under $J[2]$. The natural images of these torsors in $\HH^1(K,J)$ lie in the fiber above $\Pic^1_C$ under multiplication by $2$. This allows one to deduce the existence of a complex
	\begin{equation}\label{eq:Pic1descent}
		\frac{\Pic C}{2\Pic C} \stackrel{x-\alpha}{\To} \mathfrakL_c \stackrel{d}\To \left(\frac{\HH^1(K,J)}{\langle \Pic^1_C \rangle}\right)[2] \To 0\,,
	\end{equation}
	 which is compatible with~\eqref{eq:fakedescent}, and is exact when $\Br K[2] = 0$. To prove Theorems~\ref{thm:ComputeH1} and~\ref{thm:ExactIfC1} we must show that these maps are compatible in the sense that ${h_0} \circ \gamma = d$.

	%%%%%%%%%%%%%%%%%%%%%%%%%%%%%%%%%%%%%%%%%%%%%%%%%%%%%%%%%%%%%%%%%%%%%%%%%%%%
	%\subsection*{Outline}%%%%%%%%%%%%%%%%%%%%%%%%%%%%%%
	%%%%%%%%%%%%%%%%%%%%%%%%%%%%%%%%%%%%%%%%%%%%%%%%%%%%%%%%%%%%%%%%%%%%%%%%%%%%
		%In \S\ref{sec:gamma} we compute the residues of an algebra of the form $\gamma'(\ell)$ and use this to show that $\gamma'$ induces a complex as stated in the theorems above. Then in section \S\ref{sec:hofgamma} we define the map ${h_0}$ and compute ${h_0} \circ \gamma'$ in terms of cocycles. This is then related to the cohomological setup for $2$-descents in \S\ref{sec:2descent}. The results of the preceding sections are then utilized in \S\ref{sec:Normc} to prove the theorems above. \S\ref{sec:CTpairing} gives an arithmetic application of these results.

		%%%%%%%%%%%%%%%%%%%%%%%%%%%%%%%%%%%%%%%%%%%%%%%%%%%%%%%%%%%%%%%%%%%%%%%%
		\subsection{Notation}%%%%%%%%%%%%%%%%%%%%%%%%%%%%%%%%%%%%%%%%%%%%%%%%%%
		%%%%%%%%%%%%%%%%%%%%%%%%%%%%%%%%%%%%%%%%%%%%%%%%%%%%%%%%%%%%%%%%%%%%%%%%
			Let $K$ be a field, choose a separable closure $\Kbar$ and let $G_K := \Gal(\Kbar/K)$ be the absolute Galois group. If $M$ is a $G_K$-module (with the discrete topology) and $i \ge 0$, then $\HH^i(K,M) := \HH^i(G_K,M)$ denotes the $i$th Galois cohomology group. Similarly $C^i(K,M)$ and $Z^i(K,M)$ are, respectively, the groups of continuous $i$-cochains and $i$-cocycles. More generally, if $A$ is an \'etale $K$-algebra, then $\HH^i(A,M)$ denotes the \'etale cohomology group $\HH^i_\textup{\'et}(\Spec A,M)$. If $A \simeq \prod K_j$ for field extensions $K_j/K$, Shapiro's lemma shows that $\HH^i(A,M) \simeq \prod\HH^i(K_j,M)$. If $\calG$ is an algebraic group defined over $K$ we define $\HH^i(K,\calG) := \HH^i(K,\calG(\Kbar))$, and analogously for the other groups defined above.
			
			{The Brauer group of a scheme $X$ is the \'etale cohomology group $\Br X := \HH^2_\textup{\'et}(X,\G_m)$; when $X = \Spec R$ is the spectrum of a ring $R$ we define $\Br R := \Br \Spec R$. Given invertible elements $a,b$ in an \'etale $K$-algebra $A$, we define the quaternion algebra 
			\[
				(a,b)_2 := A[i,j]/\langle i^2 = a,\, j^2 = b,\, ij = -ji\rangle.
			\]
			We will often conflate $(a,b)_2$ with its class in $\Br A$ where we write the group law additively. If $X$ and $S$ are $K$-schemes, we set $X_S := X \times_{\Spec K} S$. We also define $\overline{X} := X_\Kbar$ and $X_A := X_{\Spec A}$, for a $K$-algebra $A$ of finite type. If $X$ is an integral $K$-scheme, $\kk(X)$ denotes its function field. More generally, if $X$ is a finite union of integral $K$-schemes $X_i$, then $\kk(X) := \prod \kk(X_i)$ is the ring of global sections of the sheaf of total quotient rings. In particular, if $A \simeq \prod K_j$ is an \'etale $K$-algebra, then $X_A$ is a union of integral $K$-schemes and $\kk(X_A) \simeq \prod\kk(X_{K_j})$.	
			
			Now suppose that $X$ is a smooth, projective and geometrically integral variety over $K$. Let $\Pic X$ be its Picard group and let $\Pic_X$ be its Picard scheme. Then $\Pic X = \Div X/\Princ X$, where $\Div X$ (resp. $\Princ X$) is the group of divisors (resp. principal divisors) of $X$ defined over $K$. If $D \in \Div X$, then $[D]$ denotes its class in $\Pic X$. There is a bijective map $(\Pic \Xbar)^{G_K} \to \Pic_X(K)$, but in general the map $\Pic X \to \Pic_X(K)$ is not surjective. Let $\Pic^0_X \subseteq \Pic_X$ denote the connected component of the identity, and use $\Pic^0X$ to denote the subgroup of $\Pic X$ mapping into $\Pic^0_X(K)$. Then $\NS X := \Pic X/\Pic^0 X$ is the {N\'eron-Severi group} of $X$. If $\lambda \in (\NS \Xbar)^{G_K}$, let $\Pic^\lambda_X$ denote the corresponding component of the Picard scheme and use $\Pic^\lambda X$ and $\Div^\lambda X$ to denote the subsets of $\Pic X$ and $\Div X$ mapping into $\Pic^\lambda_X(K)$.	We write $\Alb_X$ for the Albanese scheme of $X$ and, for $i \in \Z$, write $\Alb^i_X$ for the degree $i$ component of $\Alb_X$. Then $\Alb^i_X$ is a $K$-torsor under the abelian variety $\Alb^0_X$. When $X$ is a curve, $\NS \Xbar = \Z$, $\Pic^i_X = \Alb^i_X$ for all $i \in \Z$ and $\Jac(X) := \Pic^0_X = \Alb^0_X$ is called the {Jacobian} of $X$.
		
				%%%%%%%%%%%%%%%%%%%%%%%%%%%%%%%%%%%%%%%%%%%%%%%%%%%%%%%%%%%%%%%%%%%%%%%%%%%%
	\section{Corestriction, residues and purity}\label{sec:gamma}%%%%%%%%%%%%%%%%%%%%%%%
	%%%%%%%%%%%%%%%%%%%%%%%%%%%%%%%%%%%%%%%%%%%%%%%%%%%%%%%%%%%%%%%%%%%%%%%%%%%%
		Every prime valuation $v$ on $\kk(C)$ induces a residue map,
		\[
			\partial_v : \Br \kk(C) \To \HH^1\left(\kk(C),\Q/\Z\right)\,,  
		\]
		where $\kk(v)$ denotes the residue field associated to $v$ (see \cite[\S 6.4]{GS-csa}). By the purity theorem \cite{Fujiwara-purity} we have
		\[
			\Br C = \bigcap_v \ker\left(  \Br \kk(C) \stackrel{\partial_v}\To \HH^1\left(\kk(C),\Q/\Z\right) \right)\,,
		\]
		the intersection running over all valuations corresponding to prime divisors on $C$. 
		
		Restricting to the $2$-torsion subgroup, identifying $\mu_2$ with $\frac{1}{2}\Z/\Z$, and applying Hilbert's Theorem 90 the residue map corresponding to $v$ gives a map
		\[
			\partial_v^2 : (\Br \kk(C))[2] \To \HH^1\left(\kk(v),\mu_2\right) \simeq \kk(v)^\times/\kk(v)^{\times 2}\,,
		 \]
		 which sends the class represented by a quaternion algebra $(a,b)_2$ to
		 \[
		 	\partial_v^2\left( (a,b)_2 \right) = {\left((-1)^{v(a)v(b)}a^{v(b)} b^{-v(a)}\right)} \in \kk(v)^\times/\kk(v)^{\times 2}
		 \]
		 (see \cite[Example 7.1.5]{GS-csa}).

The following lemma will enable us to compute the residues for a Brauer class of the form $\Cor_{\kk(C_L)/\kk(C)}((\ell,x-\alpha)_2)$.
		
		\begin{lemma}
			\label{lem:CorResidue}
			Suppose $E$ is a field of characteristic prime to $2$ and that $v$ is a discrete valuation on $E$ with residue field $\kk(v)$ of equal characteristic. If $F/E$ is a field of finite degree over $E$, then for any $a,b \in F^\times$ we have
			\begin{equation*}
				\partial_v^2\left( \Cor_{F/E}((a,b)_2)\right) 
				%= \prod_{w \mid v} \Cor_{\kk(w)/\kk(v)}\left(\partial_w((a,b)_2)\right)\\
				= \prod_{w\mid v} \Norm_{\kk(w)/\kk(v)}
						\left((-1)^{w(a)w(b)}a^{w(b)} b^{-w(a)}\right)\,,
			\end{equation*}
			the product running over all discrete valuations of $F$ lying over $v$.
		\end{lemma}
		
		\begin{proof}
			Consider the following diagram:
			\begin{equation}
			\label{cube1}
			   \xymatrix {
				    K^M_2(F) 
						\ar[dd]^{\N_{F/E}} 
						\ar[rr]^{\oplus \partial^M_w} 
						\ar[dr]^{h^2_{F,2}} & & 
					\bigoplus_{w|v}K^M_1(\kk(w)) 		
						\ar[dd]|\hole^(.7){\sum_{w|v}\N_{\kk(w)/\kk(v)}} 	
						\ar[dr]^{\oplus h^1_{\kk(w),2}}\\
				    & \HH^2(F,\mu_2^{\otimes2})  
						\ar[rr]^(.3){\oplus \partial^2_w} 
						\ar[dd]^(.7){\Cor_{F/E}} & & 
					{\bigoplus_{w|v}\HH^1(\kk(w), \mu_2)} 
						\ar[dd]^{\sum_{w|v}\Cor_{\kk(w)/\kk(v)}}  \\
				    K^M_2(E)
						\ar[rr]^(.7){\partial^M_v}|\hole 
						\ar[dr]^{h^2_{E,2}} & & K^M_1(\kk(v)) 	
						\ar[dr]^{h^1_{\kk(v),2}}\\
				    & \HH^2(E, \mu_2^{\otimes2}) 
						\ar[rr]^{\partial^2_v}  & & 
					\HH^1(\kk(v), \mu_2) 
				   }
			\end{equation}
			The back, side, top, and bottom squares are all commutative~\cite[Cor. 7.4.3 and Prop. 7.5.1 \& 7.5.5]{GS-csa}.  Therefore, all ways of traversing from $K^M_2(F)$ to $\HH^1(\kk(v), \mu_2)$ are equivalent.  By the Merkurjev-Suslin theorem \cite[Thm. 4.6.6]{GS-csa}, $h^2_{F,2}$ is surjective so the front square is commutative.
			
			Hilbert's Theorem 90 yields identifications $\HH^2(E,\mu_2^{\otimes 2}) \simeq (\Br E)[2]$ and  $\HH^1(\kk(v),\mu_2) \simeq \kk(v)^\times/\kk(v)^{\times 2}$ (and similarly with $F$ and $\kk(w)$) under which the front square of \eqref{cube1} becomes
			\begin{equation}
				\label{diag:purity-cores}
				\xymatrix{
					(\Br F)[2] \ar[rrr]^{\oplus_w \partial^2_w} \ar[d]^{\Cor_{F/E}}
					&&&\bigoplus_{w\mid v} 	\kk(w)^{\times}/\kk(w)^{\times2} \ar[d]^{\Norm_{\kk(w)/\kk(v)}}\\
					(\Br E)[2] \ar[rrr]^{\partial_v^2}
					&&&\kk(v)^{\times}/\kk(v)^{\times2}\,.
				}
			\end{equation}
			The lemma now follows immediately from the explicit formula for the residue map mentioned above.
		\end{proof}

\begin{proof}[Proof of Theorem \ref{thm:gammawelldefined}]
	Let $\ell \in L^\times$. By the purity theorem and Lemma~\ref{lem:CorResidue}, we have that $\gamma'(\ell)$ lies in $\Br C$ if and only if 
				\begin{equation}\label{eqn:unramified}
					\partial_v^{2}(\gamma'(\ell)) = \prod_{w|v} \Norm_{\kk(w)/\kk(v)} \left((-1)^{w(\ell)w(x - \alpha)}
					\ell^{w(x- \alpha)} (x - \alpha)^{-w(\ell)}\right)
				\end{equation}
				is a square in $\kk(v)^{\times}$, for all valuations $v$.  Since $\ell$ is a constant in $\kk(C_L)$, $w(\ell) = 0$ for all valuations $w$. When $C$ is odd, $w(x - \alpha) \equiv 0 \bmod 2$, for all valuations $w$,  so~\eqref{eqn:unramified} is clearly a square. Hence we may assume that $C$ is even. Furthermore we can restrict our attention to valuations $v$ such that there exists a $w\mid v$ with $w(x - \alpha)\neq 0$. 
				
				For all valuations $w$ such that $w(x - \alpha)$ is positive, we have that $w(x - \alpha) \equiv 0 \bmod 2 $ so~\eqref{eqn:unramified} is clearly a square.  Thus we may consider the valuations $v$ for which there exists a $w\mid v$ with $w(x - \alpha) <0$.  Such valuations $v$ correspond to the points at infinity on $C$, and, for every $w\mid v$, we have that $w(x - \alpha) = -1$.  In this case~\eqref{eqn:unramified} can be simplified to
				\[
					\partial_v^{2}(\gamma'(\ell)) 
					= \prod_{w|v} \Norm_{\kk(w)/\kk(v)} \left(\ell^{-1}\right)
					= \Norm_{L/K}(\ell^{-1}).
				\]
				This shows that $\gamma'(\ell) \in \Br C$ if and only if $\Norm_{L/K}(\ell) \in \kk(v)^{\times 2}$. But $\kk(v) = K(\sqrt{c})$, so this is equivalent to requiring that $\overline{\ell} \in \mathfrakL_c$. This completes the proof.
			\end{proof}

			\begin{lemma}
				$\gamma'$ induces a {homomorphism} $\gamma\colon \mathfrakL \to \Br \kk(C)/\Br_0 C$ such that the natural square commutes.
			\end{lemma}
			
			\begin{proof}
				{The map $\gamma'$ is clearly a homomorphism.  It remains to show that $\gamma'$ sends $K^\times L^{\times2}$ into $\Br_0 C.$}
				Let $a \in K^\times$, $\ell \in L^{\times}$; we may expand $\gamma'(a\ell^2)$ as follows
				\begin{align*}
					\gamma'(a\ell^2)& = 
						\Cor_{\kk(C_L)/\kk(C)}\left( (a, x - \alpha)_2\right) +
						\Cor_{\kk(C_L)/\kk(C)}\left( (\ell^2, x - \alpha)_2\right)\\
					& =  (a, \Norm_{\kk(C_L)/\kk(C)}(x - \alpha))_2
					 =  (a, f(x))_2
					 =  (a, y^2/c)_2
					 = (a,c)_2\,.
				\end{align*}
				This completes the proof since $(a,c)_2 \in \Br_0 C$.
			\end{proof}
			
			Theorem~\ref{thm:MainComplex} is a corollary of the following proposition.
			
			\begin{prop}
				\label{prop:complex}
				If $P\in C \setminus\left(\Omega\cup \pi^{-1}(\infty)\right)$ is a closed point, then $\gamma'(({x(P)} - \alpha)) \in \Br_0 C$.
			\end{prop}

			\begin{proof}[Proof of Theorem~\ref{thm:MainComplex}]
				$[D] \in \Pic C$ is represented by a linear combination $\sum_{P}{n_P}P$ of points $P$ as in the proposition (see \cite[\S 5]{PS-descent}), and, by definition, $(x-\alpha)([D])$ is the class of $\prod (x(P)-\alpha)^{n_P}$ in $\mathfrakL$.
			\end{proof}
			
			\begin{proof}[Proof of Proposition~\ref{prop:complex}]
				
				Let ${p(x)}\in K[x]$ be the minimal polynomial of the $x$-coordinate of $P$. If $P$ is the pullback of a point from $\PP^1$, then $(x{(P)} - \alpha)\in L^{\times2}$ and so $\gamma'(x{(P)} - \alpha) = 0.$  Assume otherwise; then $\gamma'(x{(P)} - \alpha) 
					= \Cor_{L(x)/K(x)}\left(({(-1)^{{\deg(P)}}}{p(\alpha)}, x - \alpha)_2\right)\,;$ we note this element is in $\Br K(x) = \Br \kk(\PP^1_K)$. We will show that the algebras $\gamma'(x{(P)} - \alpha)$ and $\calA := (cf(x), {(-1)^{{\deg(P)}}}{p(x)})_2$ have the same residue at all points of $\PP^1_K$.  Since $\Br \PP^1_K = \Br K$, this shows that $\gamma'(x(P) - \alpha)$ and $\calA$ differ by a constant algebra. To complete the proof, we note that $\calA\in\ker( \Br \kk(\PP^1_K) \to \Br \kk(C))$.
				
				Considered as an element of $\Br \kk(\PP^1)$, the algebra $\gamma'(x(P) - \alpha)$ has trivial residue away from the $\infty$ and the roots of $f(x)$. The residue at $\infty$ is $\Norm_{L/K}({(-1)^{{\deg(P)}}}{p(\alpha)})$ which is equal to $c^{\deg(P)}$ in $K^{\times}/K^{\times2}.$
				
				Now we compute the residues at the roots of $f(x)$. Let $f_Q$ be an irreducible factor of $f$ corresponding to a root $Q$ of $f(x)$, and let $\beta$ be the image of $\theta$ in $\kk(Q) = K[\theta]/f_Q(\theta)$. There is a unique valuation on $\kk(Q)(x) \subseteq L(x)$ lying above $Q$ such that $({(-1)^{{\deg(P)}}}{p(\alpha)}, x - \alpha)_2$ has nontrivial residue, namely the valuation corresponding to the point $Q' = (\alpha:1)$. Furthermore, the norm map $\kk(Q') \to \kk(Q)$ is an isomorphism which sends $\alpha$ to $\beta$.  Therefore, using Lemma~\ref{lem:CorResidue}, we see that the residue at $Q$ is ${(-1)^{{\deg(P)}}}{p(\beta)}$.
				
				{Now we consider the algebra $\calA$; it has trivial residue away from $P$, $\infty$ and the zeros of $f(x)$. The residue at $P$ is equal to $cf(x(P))$, which is a square, the residue at $\infty$ is $(-1)^{\deg(P)\deg(f)}(cf(\infty))^{\deg(P)}({(-1)^{{\deg(P)}}}{p(\infty)})^{-\deg(f)} = c^{\deg(P)}$, and the residue at a zero $Q$ of $f(x)$ is ${(-1)^{{\deg(P)}}}{p(\beta)}$.  Therefore, the residues of $\calA$ and $\gamma'(x(P) - \alpha)$ are equal.}
			\end{proof}
	
		%%%%%%%%%%%%%%%%%%%%%%%%%%%%%%%%%%%%%%%%%%%%%%%%%%%%%%%%%%%%%%%%%%%%%%%%
		\subsection{Corestriction as a tensor product of quaternion algebras}%%%%%%%%%%%%%%
		%%%%%%%%%%%%%%%%%%%%%%%%%%%%%%%%%%%%%%%%%%%%%%%%%%%%%%%%%%%%%%%%%%%%%%%%
			
			{Using Rosset-Tate reciprocity, one can write the corestriction of a quaternion algebra over an extension as a sum of quaternion algebras over the base field. This is described in \cite[Corollary 7.4.10 and Remark 7.4.12]{GS-csa}. In our situation this allows us to write $\gamma'(\ell)$ as a sum of quaternion algebras over $K(x)$. We caution the reader that the $f$ and $g$ appearing in the proposition below are not to be confused with the $f$ and $g$ of \cite[Corollary 7.4.10]{GS-csa}}.
			
			\begin{prop}\label{prop:RT}
				Suppose $\ell \in L^\times \setminus K^\times$ and let $g(x) \in K[x]$ be the minimal degree polynomial such that $g(\alpha) = \ell$. Set $r_0 = f(x)$, $r_1 = g(x)$, and for $i \ge 0$ define $r_{i+2}$ to be the unique polynomial of degree less than $\deg(r_{i+1})$ such that $r_{i+2} \equiv r_{i} \bmod{r_{i+1}}$. Then
				\[
					\Cor_{\kk(C_L)/\kk(C)}\left((\ell,x-\alpha)_2\right) = {\left(\sum_{i = 0}^n (r_{i+1},r_{i})_2\right) + \left(\sum_{i = 0}^n (a_{i+1},a_{i})_2\right)\,,}
				\]
				where $a_i$ is the leading coefficient of $r_i$ and $n$ is the first integer such that $r_{n+2} = 0$.			
			\end{prop}
			
			\begin{cor}\label{cor:RT}
				Modulo constant algebras, $\gamma'(\ell)$ may be written as a sum of $g(C) + 1$ quaternion algebras over $K(x)$.
			\end{cor}
						
			\begin{proof}
				The proposition shows that, modulo constant algebras,
				\[
					\gamma'(\ell) 
					{=\underbrace{(r_1,r_0)_2 + (r_2,r_1)_2}_{= (r_1,r_0r_2)_2 }
					+ \cdots + 
					\underbrace{(r_{n},r_{n-1})_2 + (r_{n+1},r_{n})_2}_{=(r_n,r_{n-1}r_{n+1})_2}}
				\]
				is a sum of $\left\lceil n/2 \right\rceil$ quaternion algebras over $K(x)$. On the other hand, the $r_i$ are the remainders obtained by applying the Euclidean algorithm to $f(x)$ and $g(x)$, so $n \le \deg(f(x)) \le 2(g(C)+1)$.
			\end{proof}
			
			\begin{example}
				\label{ex:RLcurve}
				Suppose $C : y^2 = 2(x^4-17)$. Then $\ell := (-\alpha^2-4) \in L^\times$ has norm $1$, so $\overline\ell \in \frak{L}_1$. In the notation of Proposition~\ref{prop:RT} we have
				\[
					\underbrace{(x^4-17)}_{f(x) = r_0} = \underbrace{(-x^2-4)}_{g(x)=r_1}(-x^2+4) + \underbrace{(-1)}_{r_2}
				\]
				Hence,
				\begin{align*}
					\gamma'(\ell) 
					&= (-x^2-4,x^4-17)_2 + (-1,-x^2-4)_2\\
					&= (-x^2-4,2)_2 + (-1,-x^2-4)_2\\
					&= (-x^2-4,-2)_2\,,
				\end{align*}
				as in the introduction.
			\end{example}
			
			\begin{proof}[Proof of Proposition~\ref{prop:RT}]
				For $i \ge 0$, let $R_i(y) = r_i(x+y)$, considered as an element in the Euclidean ring $K(x)[y]$. Then, for all $i\ge 0$, the leading coefficient of $R_i(y)$ is $a_i$, and
				\[
					R_{i+2}(y) \equiv R_{i}(y) \bmod {R_{i+1}(y)}\,.
				\]
				Moreover, $R_0(-x+\alpha) = f(\alpha) = 0$ and $R_1(-x+\alpha) = g(\alpha) = \ell$, and $R_i$ and $R_j$ are relatively prime for all $i,j$ since $\ell\in L^{\times}$.  In particular, $r_{n+1}$ and $R_{n+1}$ are nonzero constants. So by~\cite[Lemma 7.4.6 and Proposition 7.5.5]{GS-csa},
				\[
					\Cor_{L(x)/K(x)}\left((\ell,x-\alpha)_2\right) =
					(R_1(y)|R_0(y))_\textsc{rt}\,,
				\]
				where $(\cdot|\cdot)_\textsc{rt}$ denotes the Rosset-Tate symbol. For any $i \ge 0$, the Rosset-Tate reciprocity law \cite[Theorem 7.4.9]{GS-csa} and the Merkurjev-Suslin theorem \cite[Theorem 4.6.6]{GS-csa} give
				\begin{align*}
					(R_{i+1}(y)|R_{i}(y))_\textsc{rt} 
					&= (R_{i+2}(y)|R_{i+1}(y))_\textsc{rt} + (R_{i+1}(0),R_{i}(0))_2 + (a_{i+1},a_{i})_2\\
					&= (R_{i+2}(y)|R_{i+1}(y))_\textsc{rt} + (r_{i+1},r_{i})_2 + (a_{i+1},a_{i})_2\,.
				\end{align*}
				From this the result easily follows by induction.
			\end{proof}
			
		%%%%%%%%%%%%%%%%%%%%%%%%%%%%%%%%%%%%%%%%%%%%%%%%%%%%%%%%%%%%%%%%%%%%%%%%
		\subsection{When $C$ is odd}%%%%%%%%%%%%%%%%%%%%%%%%%%%%%%%%%%%%%%%%%%%%
		%%%%%%%%%%%%%%%%%%%%%%%%%%%%%%%%%%%%%%%%%%%%%%%%%%%%%%%%%%%%%%%%%%%%%%%%
		
			\begin{lemma}
				\label{lem:Br^0}
				Suppose that $C$ is odd. Then, for every $\ell \in L^\times$, $\gamma'(\ell)$ evaluates to $0$ at the point $\infty_C \in C(K)$ above $\infty \in \PP^1(K)$.
			\end{lemma}
			
			\begin{proof}
				Since $\deg f(x)$ is odd, the functions $(x-\alpha)$ and $\frac{(x-\alpha)^{\deg f(x)}}{y^2}$ represent the same class in $\kk(C_L)^\times/\kk(C_L)^{\times 2}$. The latter evaluates to $1$ at $\infty_C$, from which it follows that $\gamma'(\ell)$ is trivial at $\infty_C$.
			\end{proof}
			
			\begin{remark}
				\label{rem:NeedHypInf}
				For this lemma it is not enough to assume the existence of a rational ramification point; one must in fact have an odd double cover. For example, suppose $C$ is defined by $y^2 = x(x-a_1)(x-a_2)(x-a_3)$ with $a_i \in K^\times$. Then $\mathfrakL_1 \simeq (K^\times/K^{\times 2})\times (K^\times/K^{\times 2})$ and $\gamma'$ sends $(k_1,k_2) \in K^\times\times K^\times$ to {$(k_1,(x - a_1)(x- a_3))_2 + (k_2, (x - a_2)(x - a_3))_2$}. Evaluating at the ramification point $\omega = (0,0)$ we have the algebra {$(k_1,a_1a_3)_2 + (k_2,a_2a_3)_2$}. The only conditions these must satisfy are $k_i,a_i \in K^\times$ and {that the $a_i$ are distinct}.  {Over say, $K=\Q$, one can easily find $k_i,a_i$ for which this algebra is nontrivial.}
			\end{remark}

	%%%%%%%%%%%%%%%%%%%%%%%%%%%%%%%%%%%%%%%%%%%%%%%%%%%%%%%%%%%%%%%%%%%%%%%%%%%%
	\section{Computation of cocycles}\label{sec:hofgamma}%%%%%%%%%%%%%%%%%%%%%
	%%%%%%%%%%%%%%%%%%%%%%%%%%%%%%%%%%%%%%%%%%%%%%%%%%%%%%%%%%%%%%%%%%%%%%%%%%%% 
		Consider the following diagram:
		\begin{equation}
			\label{diag:defineh}
			\xymatrix{
				\Br K \ar[r]\ar@{=}[d] &
				\Br C \ar[r]^h\ar@^{^{(}->}[d]^\phi&
				\HH^1(K,\Pic_C) \ar[r]\ar@^{^{(}->}[d]^\rho &
				\HH^3(K,\G_m)\ar@{=}[d]\\
				\Br K \ar[r]&
				\HH^2(K,\kk(\Cbar)^\times) \ar[r]\ar[d]^{\divv_*} \ar[r]^{j_*}&
				\HH^2(K,\kk(\Cbar)^\times/\Kbar^\times) \ar[r]\ar[d]^{\divv_*}&
				\HH^3(K,\G_m)\\
				&\HH^2(K,\Div\Cbar) \ar@{=}[r]&
				\HH^2(K,\Div\Cbar)
			}
		\end{equation}
		
		We claim that this diagram is commutative and that all rows and columns are exact. The existence and exactness of the morphisms in the top row can be deduced from exactness in the rest of the diagram. The second row and column come, respectively, from the Galois cohomology of the exact sequences,
		\begin{align*}
			&1 \to \Kbar \to \kk(\Cbar)^\times \stackrel{j}\to \kk(\Cbar)^\times/\Kbar^\times \to 1\,,\\
			\intertext{and}
			&1 \to \kk(\Cbar)^\times/\Kbar^\times \stackrel{\divv}\to \Div\Cbar \to \Pic\Cbar \to 0 \,.
		\end{align*}
		The connecting homomorphism $\rho$ is injective since $\Div\Cbar$ is a permutation module, which by Shapiro's lemma implies that $\HH^1(K,\Div\Cbar) = 0$. By Tsen's theorem the inflation map 
			\[
				\inf:\HH^2(K,\kk(\Cbar)^\times) \to \HH^2(\kk(C),\overline{\kk(C)}^\times) = \Br \kk(C)
			\]
			is an isomorphism. {The map $\phi$ is the composition of the inverse of this inflation map with the inclusion $\Br C \subset \Br \kk(C)$. Exactness of the first column is proven in \cite[Lemme 14]{CTSan}}. Commutativity of the bottom square is obvious. The other squares commute by definition, so the diagram is exact and commutative as claimed.
		
		\begin{remark}
			The existence of an exact sequence as in the top row of~\eqref{diag:defineh} also follows from the spectral sequence $\HH_{\textup{\'et}}^p(K,\HH_{\textup{\'et}}^q(\Cbar,\G_m)) \Rightarrow \HH^n_{\textup{\'et}}(C,\G_m)$. One can check that these coincide, at least up to sign. See \cite[Annexe]{CTSan}.
		\end{remark}
		
		{It follows from the definition of $\phi$ that} the map $\Br K \to \Br C$ in the top row of~\eqref{diag:defineh} is the natural map induced by the structure morphism of $C$. Hence, the map $h$ in the top row induces an injective homomorphism $h_0: \Br C/\Br_0C \to \HH^1(K,\Pic_C)$. The goal of this section is to compute the composition

		\begin{equation}
			\mathfrakL_c \stackrel{\gamma}\To \frac{\Br C}{\Br_0C} \stackrel{h_0}\To \HH^1(K,\Pic_C)
		\end{equation}
		explicitly. This is accomplished in Proposition \ref{prop:xi} below, but first we need to fix some notation. 
		
		Given $\ell \in L^\times$, let $\chi_\ell \in Z^1(K,\mu_2(\Lbar))$ be the corresponding quadratic character, i.e., fix a square root $m \in \Lbar^\times$ of $\ell$, and define $\chi_\ell(\sigma) = \sigma(m)/m$. Composing $\chi_\ell$ with the bijection $\mu_2 \to \{ 0,1\} \subseteq \Z$ sending $-1$ to $1$, we obtain a map $\tilde\chi_\ell \in C^1(K,\Z^\Omega)$. For any $\tau \in G_K$, we may consider $\tilde\chi_\ell(\tau)$ as a map $\Omega \to \{0,1\} \subseteq \Z$ whose value at $\omega \in \Omega$ will be denoted $\tilde\chi_\ell(\tau)_\omega$.  Note that the action of an element $\sigma\in G_K$ on the map $\tilde\chi_\ell(\tau)$ is then given by $\sigma(\tilde\chi_\ell(\tau))_\omega = \tilde\chi_{\ell}(\tau)_{\sigma^{-1}\omega}$. The norm of $\chi_\ell$ is the quadratic character $\chi_a \in Z^1(K,\mu_2(\Kbar))$ associated to $a = \Norm_{L/K}(\ell) \in K^\times$. We let $\tilde\chi_a \in C^1(K,\Z)$ denote the corresponding map to $\{0,1\}$. We can then define a $1$-cochain $g_\ell\in C^1(K,\Z)$ by requiring that 
		\begin{equation}
			\label{defineg}
			\sum_{\omega \in \Omega}\tilde\chi_\ell(\sigma)_\omega = 2g_\ell(\sigma) + \tilde\chi_{a}(\sigma)\,,\text{ for all $\sigma \in G_K$}\,.
		\end{equation}
		When $C$ is even we use $\infty^+$ and $\infty^-$ to denote the points on $C$ lying above $\infty \in \PP^1$. When $C$ is odd we use both $\infty^+$ and $\infty^-$ to denote the unique point $\infty_C \in C(K)$ lying above $\infty \in \PP^1(K)$. In both cases we set  $\mm = (\infty^+ + \infty^-) \in \Div C$. 
	
		\begin{prop}\label{prop:xi}
		 	Let $\xi_\ell \in C^1(K,\Pic_C)$ be the $1$-cochain defined by
			\begin{equation}\label{eq:defxi}
				\xi_\ell(\sigma) = \left(\sum_{\omega \in \Omega} \tilde\chi_\ell(\sigma)_\omega[\omega]\right) - {g_{\ell}(\sigma)}[\mm] - \tilde{\chi}_a(\sigma)[\infty^+]\,.
			\end{equation}
			If $\ell$ represents a class in $\mathfrakL_c$, then 
			\begin{enumerate}
				\item\label{it:xi1} $\xi_\ell$ is a cocycle, and
				\item\label{it:xi2} the image of $\xi_\ell$ in $\HH^1(K,\Pic_C)$ is equal to $(h \circ \gamma')(\ell)$.
	
			\end{enumerate}
		\end{prop}
	
		To prove Proposition~\ref{prop:xi}\eqref{it:xi2} we will explicitly compute the images of $\xi_\ell$ and $\gamma'(\ell)$ under the maps $\phi$ and $\rho$ of diagram~\eqref{diag:defineh}. This will involve a rather technical computation with cocycles carried out in the lemmas below. Having accomplished this, the proposition will follow from a simple diagram chase. 
		
		\begin{lemma}\label{lem:sigmatau1}
			For any $\sigma, \tau \in G_K$ and $\omega \in \Omega$ we have
			\begin{enumerate}
				 \item\label{it:1} $\tilde\chi_\ell(\tau)_{\sigma^{-1}\omega} + \tilde\chi_\ell(\sigma)_\omega  - \tilde\chi_\ell(\sigma\tau)_\omega = 2\tilde\chi_\ell(\sigma)_\omega\tilde\chi_\ell(\tau)_{\sigma^{-1}\omega}$, and
				\vspace{.05in}
				 \item\label{it:2} $g_{\ell}(\tau) + g_{\ell}(\sigma) - g_{\ell}(\sigma\tau) +  \tilde{\chi}_a(\sigma)\tilde{\chi}_a(\tau) =  \#\{ \omega \in \Omega \,:\, \tilde\chi_\ell(\sigma)_\omega\tilde\chi_\ell(\tau)_{\sigma^{-1}\omega} = 1\}$.
			\end{enumerate}
			If, moreover, $\ell$ represents a class in $\mathfrakL_c$, then 
			\begin{enumerate}[resume]
				 \item\label{it:3} $\sigma(\tilde\chi_a(\tau)\infty^+) + \tilde\chi_a(\sigma)\infty^+ -  \tilde\chi_a(\sigma\tau)\infty^+ = \tilde\chi_a(\sigma)\tilde\chi_a(\tau)\mm\,.$
			\end{enumerate}
		\end{lemma}
		
		\begin{proof}
			Since $\chi_\ell$ is a $1$-cocycle, we have $\chi_\ell(\sigma\tau) = \sigma(\chi_\ell(\tau))\chi_\ell(\sigma)$. Evaluating at $\omega$ and rearranging we get $\chi_\ell(\tau)_{\sigma^{-1}\omega} = \chi_\ell(\sigma\tau)_\omega/\chi_\ell(\sigma)_\omega$. From this it follows that 
			\[ 
				\tilde\chi_\ell(\tau)_{\sigma^{-1}\omega} \equiv \tilde\chi_\ell(\sigma)_\omega - \tilde\chi_\ell(\sigma\tau)_\omega \bmod 2\,. 
			\] 
			Since all of the terms are either $0$ or $1$ we see that 
			\begin{equation*}
				\tilde\chi_\ell(\tau)_{\sigma^{-1}\omega} + \tilde\chi_\ell(\sigma)_\omega - \tilde\chi_\ell(\sigma\tau)_\omega = 
				\begin{cases}
					2 & \text{if $\tilde\chi_\ell(\sigma)_\omega = \tilde\chi_\ell(\tau)_{\sigma^{-1}\omega} = 1\,,$}\\
					0 & \text{otherwise\,.} 
				\end{cases}
			\end{equation*}
			This proves~\eqref{it:1}. To prove~\eqref{it:2} we sum both sides of~\eqref{it:1} over all $\omega \in \Omega$ and apply~\eqref{defineg}. This gives 
			\[
				2g_\ell(\tau) + 2g_\ell(\sigma) - 2g_\ell(\sigma\tau) + \tilde\chi_a(\tau) + \tilde\chi_a(\sigma) - \tilde\chi_a(\sigma\tau) = 2  \#\{ \omega \in \Omega \,:\, \tilde\chi_\ell(\sigma)_\omega\tilde\chi_\ell(\tau)_{\sigma^{-1}\omega} = 1\}\,.
			\] 
			Using that
			\[
				\tilde\chi_a(\tau) + \tilde\chi_a(\sigma)  - \tilde\chi_a(\sigma\tau) = 2 \tilde{\chi}_a(\sigma)\tilde{\chi}_a(\tau)\,
			\]
			(which is proved by the same argument as above), and then removing the common factor of $2$ gives~\eqref{it:2}.
			
			If $\ell$ represents a class in $\mathfrakL_c$, then $a \in c^rK^{\times 2}$ for some $r \in \{0,1\}$. If $a \in K^{\times 2}$ then both sides of~\eqref{it:3} are trivial, so to prove~\eqref{it:3} we may assume $a \in cK^{\times 2}$. Under this assumption, the action of $G_K$ on $\infty^+$ is determined by the character $\chi_a$, so all of the terms in~\eqref{it:3} are determined by the values of $\tilde\chi_a(\sigma)$ and $\tilde\chi_a(\tau)$. In each of the four possibilities, one can check directly that~\eqref{it:3} holds. This completes the proof.
		\end{proof}
		
		\begin{lemma}\label{lem:coboundarycomp}
			Assume that $\ell$ represents a class in $\mathfrakL_c$, let $\xi_\ell' \in C^1(K,\Div\Cbar)$ denote the $1$-cochain defined by
			\[
				\xi_\ell'(\sigma) = \left(\sum_{\omega \in \Omega} \tilde\chi_\ell(\sigma)_\omega\omega\right) - g_\ell(\sigma)\mm - \tilde{\chi}_a(\sigma)\infty^+\,, 
			\]
			and let $\partial:C^1(K,\Div\Cbar) \to C^2(K,\Div\Cbar)$ denote the coboundary map on cochains.  Then for $(\sigma,\tau) \in G_K \times G_K$, we have
			\[
				\partial\xi_\ell'(\sigma,\tau) = \divv \left( \Norm_{\kk(C_L)/\kk(C)}\left( (x - \alpha)^{\tilde\chi_\ell(\sigma)\cdot \sigma(\tilde\chi_\ell(\tau))}\right) \right)\,.
			\]
			In particular, $\xi_\ell$ is a cocycle and the image of its class under $\rho$ is represented by the $2$-cocycle
			\[
				(\sigma,\tau) \mapsto \Norm_{\kk(C_L)/\kk(C)}((x - \alpha)^{\tilde{\chi}_{\ell}(\sigma)\cdot\sigma(\tilde{\chi}_{\ell}(\tau))}).
			\]
		\end{lemma}
			
		\begin{proof}
			The second statement follows easily from the first.
				
			To prove the first statement we compute $\partial\xi_\ell'$ explicitly. For $(\sigma,\tau) \in G_K\times G_K$ we have
			\begin{align}\label{partialxi1}
				(\partial\xi_\ell')(\sigma,\tau) 
				& = \sum_{\omega \in \Omega} \left(\sigma(\tilde\chi_\ell(\tau)_\omega\omega) + \tilde\chi_\ell(\sigma)_\omega\omega - \tilde\chi_\ell(\sigma\tau)_\omega\omega\right)\\
				& \quad\quad - (g_\ell(\tau) + g_\ell(\sigma) - g_\ell(\sigma\tau))\mm\\ \label{partialxi2}
				& \quad\quad - \left(\sigma(\tilde\chi_a(\tau)\infty^+) + \tilde\chi_a(\sigma)\infty^+ - \tilde\chi_a(\sigma\tau)\infty^+\right)\,.
			\end{align}
			Noting that $\sum_{\omega \in \Omega} \tilde\chi_\ell(\tau)_\omega \sigma(\omega) = \sum_{\omega \in \Omega} \tilde\chi_\ell(\tau)_{\sigma^{-1}\omega} \omega$ 
			and applying Lemma \ref{lem:sigmatau1}\eqref{it:1},~\eqref{partialxi1} can be reduced to $\sum_{\omega \in \Omega} \tilde\chi_\ell(\sigma)_\omega\tilde\chi_\ell(\tau)_{\sigma^{-1}\omega}2\omega$. Lemma \ref{lem:sigmatau1}\eqref{it:3} states that~\eqref{partialxi2} is equal to $-\tilde\chi_a(\sigma)\tilde\chi_a(\tau)\mm$. Using these facts and then applying Lemma \ref{lem:sigmatau1}\eqref{it:2} we obtain,
			\begin{align*}
			(\partial\xi_\ell')(\sigma,\tau) 
				=&\left(\sum_{\omega \in \Omega} \tilde\chi_\ell(\sigma)_\omega\tilde\chi_\ell(\tau)_{\sigma^{-1}\omega}2\omega\right) - \left(g_\ell(\tau) + g_\ell(\sigma) - g_\ell(\sigma\tau) + \tilde\chi_a(\sigma)\tilde\chi_a(\tau)\right)\mm\\
				=& \left(\sum_{\omega \in \Omega} \tilde\chi_\ell(\sigma)_\omega\tilde\chi_\ell(\tau)_{\sigma^{-1}\omega}2\omega\right)
					- \#\left\{ \omega \in \Omega \,:\, \tilde\chi_\ell(\sigma)_\omega\tilde\chi_\ell(\tau)_{\sigma^{-1}\omega} = 1 \right\}\mm\\
				=& \sum_{\omega \in \Omega} \tilde\chi_\ell(\sigma)_\omega\tilde\chi_\ell(\tau)_{\sigma^{-1}\omega}\left(2\omega-\mm\right)\\
				=& \sum_{\omega \in \Omega} \divv\left( (x-x(\omega))^{ \tilde\chi_\ell(\sigma)_\omega\tilde\chi_\ell(\tau)_{\sigma^{-1}\omega} }\right)\\
				=& \divv\left( \Norm_{\kk(C_L)/\kk(C)}\left( (x - \alpha)^{\tilde\chi_\ell(\sigma)\cdot\sigma(\tilde\chi_\ell(\tau))}\right) \right)\,.
			\end{align*}
			This completes the proof.		
		\end{proof}
	
		\begin{lemma}\label{lem:cupprod}
			Let $\epsilon \in C^2(K,\kk(\Cbar)^\times)$ be the $2$-cochain defined by
			\[
				\epsilon(\sigma,\tau) = \Norm_{\kk(C_L)/\kk(C)}\left((x-\alpha)^{\tilde\chi_\ell(\sigma)\cdot\sigma(\tilde\chi_\ell(\tau))}\right)\,.
			\]
			Then $\epsilon$ is a $2$-cocycle and the map $\phi$ in~\eqref{diag:defineh} sends $\gamma'(\ell)$ to the class of $\epsilon$ in $\HH^2(K,\kk(\Cbar)^\times)$.
		\end{lemma}
	
		\begin{proof}
			The composition ${\inf \circ \phi} : \Br C \to \Br \kk(C)$ is the natural inclusion. If $\bar{\epsilon}$ denotes the cohomology class of $\epsilon$, then $\inf(\bar\epsilon)$ is represented by the cocycle $\epsilon_{\kk(C)}$ defined by
			\[
				\epsilon_{\kk(C)}(\sigma,\tau) = \Norm_{\kk(C_L)/\kk(C)}\left((x-\alpha)^{\tilde\psi_\ell(\sigma)\cdot\sigma(\tilde\psi_\ell(\tau))}\right)\,,
			\]
			where $\tilde\psi_\ell \in C^1\left(\kk(C),\{0,1\}^\Omega\right)$ and $\psi_\ell \in Z^1\left(\kk(C) , \mu_2(\Lbar)\right)$ denote the lifts of $\tilde\chi_\ell$ and $\chi_\ell$, obtained by considering $\ell$ as an element of $\kk(C_L)$. We want to show that $\epsilon_{\kk(C)}$ represents $\Cor_{\kk(C_L)/\kk(C)}\left((\ell,x-\alpha)_2\right)$. We will instead show that $\Cor_{\kk(C_L)/\kk(C)}\left((x-\alpha,\ell)_2\right)$ is represented by the inverse of $\epsilon_{\kk(C)}$. The result then follows from standard properties of the cup product (or because all elements in question are $2$-torsion).
			
			Standard cohomological arguments combined with Shapiro's lemma give a sequence of isomorphisms 
			\[
				\Br\kk(C_L)[2] \simeq \HH^2\left(\kk(C_L),\mu_2^{\otimes 2}\right) \simeq \HH^2\left(\kk(C),\mu_2(\Lbar)^{\otimes 2}\right) \simeq \HH^2\left(\kk(C),\mu_2(\Lbar)\right)\,,
			\]
			under which $(x-\alpha,\ell)_2$ is represented by the cup product, $\left(\psi_{x-\alpha} \cup \psi_\ell\right) \in Z^2(\kk(C),\mu_2(\Lbar)^{\otimes2})$. Here $\psi_{x-\alpha}$ denotes the quadratic character $\psi_{x-\alpha} \in Z^1\left(\kk(C),\mu_2(\Lbar)\right)$ associated to $x-\alpha$,{ i.e., if $s \in \overline{\kk(C)}_L^\times := (\overline{\kk(C)}\otimes_K L)^\times$ is a square root of $x-\alpha$,} then $\psi_{x-\alpha}(\sigma) = \sigma(s)/s$. The image in $\HH^2(\kk(C),\mu_2(\Lbar))$ of the cup product above is represented by the $2$-cochain,
			\begin{align*}
				\left(\psi_{x-\alpha} \cup \psi_\ell\right)(\sigma,\tau)
				&= \psi_{x-\alpha}(\sigma)\otimes\sigma(\psi_\ell(\tau))\\
				&= \left(\frac{\sigma(s)}{s}\right)^{\sigma(\tilde\psi_\ell(\tau))}
				= \frac{\sigma(s^{\tilde{\psi}_\ell(\tau)})}{s^{\sigma(\tilde{\psi}_\ell(\tau))}}\\
				&= \left(\frac{\sigma(s^{\tilde\psi_\ell(\tau)})s^{\tilde\psi_\ell(\sigma)}}{s^{\tilde\psi_\ell(\sigma\tau)}}\right)
				\left( \frac{s^{\tilde\psi_\ell(\sigma\tau)}}{s^{\sigma(\tilde{\psi_\ell}(\tau))}s^{\tilde\psi_\ell(\sigma)}} \right)\,.
			\end{align*}
			We now note that the first factor is the coboundary of the $1$-cochain 
			\[
				\left(\sigma \mapsto s^{\tilde\psi(\sigma)}\right) \in C^1\left(\kk(C),\overline{\kk(C)}_L^\times\right)\,,
			\]
			while using the obvious analog of Lemma \ref{lem:sigmatau1}\eqref{it:1} we can rewrite the second factor as
			\[
				(x-\alpha)^{-\tilde\psi_\ell(\sigma)\cdot\sigma(\tilde\psi_\ell(\tau))}\,.
			\]
			The norm of this expression is the inverse of $\epsilon_{\kk(C)}$. This proves that $\epsilon$ is a cocycle, and that $\epsilon_{\kk(C)}$ represents $\Cor_{\kk(C_L)/\kk(C)}((x-\alpha,\ell)_2)$ as required.
		\end{proof}
		
		\begin{proof}[Proof of Proposition \ref{prop:xi}]
			The first statement was proven in Lemma~\ref{lem:coboundarycomp}. For the second statement, suppose $\ell$ represents a class in $\mathfrakL_c$ and let $\bar{\xi}_\ell$ denote the class of $\xi_\ell$ in $\HH^1(K,\Pic_C)$. Lemmas~\ref{lem:coboundarycomp} and~\ref{lem:cupprod} show that
			\[
				(j_*\circ \phi \circ \gamma')(\ell) = \rho(\bar\xi_{\ell}).
			\]
			Since $\rho\circ h = j_*\circ \phi$ by~\eqref{diag:defineh} and $\rho$ is injective, this completes the proof.
		\end{proof}
			
	%%%%%%%%%%%%%%%%%%%%%%%%%%%%%%%%%%%%%%%%%%%%%%%%%%%%%%%%%%%%%%%%%%%%%%%%%%%%
	\section{Cohomological setup for $2$-descent}\label{sec:2descent}%%%%%%%%%%%
	%%%%%%%%%%%%%%%%%%%%%%%%%%%%%%%%%%%%%%%%%%%%%%%%%%%%%%%%%%%%%%%%%%%%%%%%%%%%
		
		In the previous section we explicitly computed the map $h\circ \gamma : \mathfrakL_c \to \HH^1(K,\Pic_C)$. In this section we relate this to the map $\mathfrakL_1 \to \HH^1(K,J)/\langle \Pic^1_C\rangle$ coming from the theory of explicit $2$-descents described in \cite{PS-descent}.
		
		The $2$-torsion subgroup of $J(\Kbar)$ may be identified (as a Galois module) with the set of even cardinality subsets of $\Omega$, modulo complements. Under this identification addition in $J[2]$ is given by the symmetric difference (i.e., the union of the sets minus their intersection), and the Weil pairing, denoted $e_2$, of two subsets is given by the parity of their intersection. By convention, for any $\omega\in\Omega$, the notation $\{\omega,\omega\}$ will be understood to mean the identity element.

		For any $\omega_0 \in \Omega$, we may define a map 
		\begin{equation}
			\label{map:definee}
			e_{\omega_0}: J[2] \to \mu_2(\Lbar) = \Map(\Omega, \mu_2(\Kbar)), \quad
			P \mapsto \left( \omega \mapsto e_2(P,\{\omega, \omega_0\})\right)\,.
		\end{equation}
		If $\omega_1 \in \Omega$, then, for every $P \in J[2]$, $e_{\omega_0}(P)$ and $e_{\omega_1}(P)$ differ by an element of $\mu_2(\Kbar) \subseteq \mu_2(\Lbar)$, namely the constant map $\omega\mapsto e_2(P, \{\omega_0,\omega_1\})$.  Therefore we obtain a map $e\colon J[2] \to \mu_2(\Lbar)/\mu_2(\Kbar)$ that is independent of the choice of $\omega_0\in\Omega.$  Nondegeneracy and Galois equivariance of the Weil pairing show that $e$ is an injective morphism of $G_K$-modules. On the other hand, $\sum_{\omega \in \Omega} \{ \omega, \omega_0 \} = 0 \in J[2]$. So $e$ fits into a short exact sequence, 
		\begin{equation}\label{eq:definee}
			0 \to J[2] \stackrel{e}\to \mu_2(\Lbar)/\mu_2(\Kbar) \stackrel{\Norm_{L/K}}\longrightarrow \mu_2(\Kbar) \to 1\,.
		\end{equation}
		
		\begin{remark}\label{rem:oddSES}
			When $C$ is odd we may take $\omega_0$ to be the ramification point $\infty_C \in C(K)$. The identification of $L_\circ \subseteq L$ as the subalgebra of elements taking the value $1$ at $\infty_C$ then induces a canonical isomorphism of short exact sequences of $G_K$-modules:		\[
			\xymatrix{
				0\ar[r] 
				&J[2] \ar@{^{(}->}[r]^{e_{\infty_C}} \ar@{=}[d]
				&\mu_2(\Lbar_\circ) \ar[rr]^{\Norm_{L_\circ/K}}\ar[d]^{\cong}
				&&\mu_2(\Kbar) \ar@{=}[d] \ar[r]
				&1\\
				0\ar[r]
				&J[2] \ar@{^{(}->}[r]^{e} 
				&\frac{\mu_2(\Lbar)}{\mu_2(\Kbar)}\ar[rr]^{\Norm_{L/K}}
				&&\mu_2(\Kbar) \ar[r]
				&1\,.
			}
		\]
		\end{remark}
		
		Applying Galois cohomology to~\eqref{eq:definee} gives an exact sequence,
		\begin{equation}\label{eq:estar}
			\xymatrixcolsep{2pc}\xymatrix{
			\mu_2(K) \ar[r] &\HH^1(K,J[2]) \ar[r]^-{e_*}& 
			\HH^1(K,\mu_2(\Lbar)/\mu_2(\Kbar)) \ar[rr]^-{(\Norm_{L/K})_*}&& 
			\HH^1(K,\mu_2) \,.
			}
		\end{equation} 
		
		If $D \in \Div^1(\overline{C})$ is any divisor of degree $1$ on $C$, then the $1$-cocycle sending $\sigma \in G_K$ to $[\sigma(D) - D] \in \Pic^0(\overline{C}) = J(\Kbar)$ represents the class in $\HH^1(K,J)$ of the torsor $\Pic^1_C$ parameterizing divisor classes of degree $1$. Choosing $D = \omega$ for some $\omega \in \Omega$ gives a cocycle taking values in $J[2]$, whose class in $\HH^1(K,J[2])$ does not depend on the choice for $\omega$. We will abuse notation slightly by denoting this class in $\HH^1(K,J[2])$ also by $\Pic^1_C$. One can then check that $-1$ maps to $\Pic^1_C$ under the map $\mu_2(K) \to \HH^1(K,J[2])$ in \eqref{eq:estar} (cf. \cite[Lemma 9.1]{PS-descent}).
		
		\begin{lemma}\label{lem:oddpartition}
			The following are equivalent:
			\begin{enumerate}
				\item The class of $\Pic^1_C$ in $\HH^1(K,J[2])$ is trivial.
				\item \label{oddpartition}$\Omega$ admits an unordered $G_K$-stable partition into two sets of odd cardinality.
				\item $[\mm] \in 2\Pic_C(K)$.
			\end{enumerate}
		\end{lemma}
		
		\begin{proof}
			See \cite[Lemma 11.2]{PS-descent}
		\end{proof}
		
		\begin{remark}
			Note that these equivalent conditions are trivially satisfied when $C$ is odd. When $C$ is even they occur if and only if $f(x)$ has a factor of odd degree or if the genus of $C$ is even and there exists a quadratic extension $F$ of $K$ such $f(x)$ is the norm of a polynomial in $F[x]$.
		\end{remark}

		Combining~\eqref{eq:estar} with the Galois cohomology of
		\[
			1 \to \mu_2(\Kbar) \to \mu_2(\Lbar) \stackrel{q}\to \mu_2(\Lbar)/\mu_2(\Kbar) \to 1\,.
		\]
		we obtain a commutative diagram with exact rows and columns,
		\begin{equation}\label{diag:fakedescentsetup}
			\xymatrix{
				&\frac{L^\times}{K^\times L^{\times2}} \ar[rr]^{\Norm_{L/K}}\ar@^{^{(}->}[d]^{q_*}
				&&\frac{K^\times}{K^{\times 2}}\ar@{=}[d]\\
				\frac{\HH^1(K,J[2])}{\langle \Pic^1_C \rangle} \ar@^{^{(}->}[r]^{e_*} \ar[dr]_{\Upsilon}
				&\HH^1\left(K,\frac{\mu_2(\Lbar)}{\mu_2(\Kbar)}\right) \ar[rr]^{(\Norm_{L/K})_*}\ar[d]
				&& {\HH^1(K, \mu_2)}
				\\
				& \Br K[2]
			}
		\end{equation}
		The map labelled $\Upsilon$ sends $\xi \in \HH^1(K,J[2])$ to the image of $\xi \cup \Pic^1_C$ under the map 
		\[
			\HH^1(K,J[2]\otimes J[2]) \to \HH^2(K,\mu_2) = \Br K[2]
		\] induced by the Weil pairing \cite[Proposition 10.3]{PS-descent}. Exactness at the central term of~\eqref{diag:fakedescentsetup} implies the existence of an exact sequence
		\begin{equation}
			\label{eq:define_d}
			1 \to \mathfrakL_1 \stackrel{d}\to \frac{\HH^1(K,J[2])}{\langle \Pic^1_C \rangle} \stackrel{\Upsilon} \to \Br K[2]\,.
		\end{equation}
		The exact sequence of $K$-group schemes
		\[
			0 \to \Pic^0_C \to \Pic_C \stackrel{\deg}{\to} \Z \to 0
		\]
		induces an isomorphism $\HH^1(K,J)/\langle \Pic^1_C\rangle \simeq \HH^1(K,\Pic_C)$. So, composing with $d$, the inclusions $J[2] \subseteq J = \Pic^0_C \subseteq \Pic_C$ induce maps from $\mathfrakL_1$ to $\HH^1(K,J)/\langle \Pic^1_C\rangle$ and to $\HH^1(K,\Pic_C)$. By abuse of notation we will use $d$ to denote any of these three maps. 
		The following proposition, due to Poonen and Schaefer, relates~\eqref{eq:define_d} to the Kummer sequence~\eqref{eq:Kummer}.
		\begin{prop}
			\label{prop:descent}
			The composition $d\circ(x-\alpha)$ and the {connecting homomorphism} $\delta$ in~\eqref{eq:Kummer} define the same map $\Pic^0C \to \HH^1(K,J[2])/\langle \Pic^1_C\rangle$.
		\end{prop}
		
		\begin{cor}
			\label{cor:descent}
			There is an exact sequence
			\[
				\Pic^0C \stackrel{x-\alpha}\To \mathfrakL_1 \stackrel{d}\to \HH^1(K,\Pic_C)\,.
			\]
		\end{cor}
		
		\begin{proof}[Proof of Proposition~\ref{prop:descent}]
			See \cite[Theorem 9.4]{PS-descent} when $C$ is even and \cite[Theorem 1.1]{Schaefer-descent} when $C$ is odd (see Remark~\ref{rem:oddSES}).	
		\end{proof}
		
		The following lemma gives an explicit description of the map $d$.
		\begin{lemma}
			\label{lem:described'}
			Suppose $\ell \in L^\times$ represents a class $\overline{\ell} \in \mathfrakL_1$ and let $\tilde{\chi}_\ell, g_\ell \in C^1(K,\Z)$ be as in~\eqref{defineg}. Then $d(\overline{\ell})$ is represented by the $1$-cocycle $\xi''_\ell \in Z^1(K,J[2])$ defined by
			\[
				\xi''_\ell(\sigma) = \left(\sum_{\omega \in \Omega}\tilde\chi_\ell(\sigma)_\omega[\omega]\right) - g_\ell(\sigma)[\mm]\\.
			\]
		\end{lemma}
		
		\begin{proof}
			The map $q_*$ in diagram~\eqref{diag:fakedescentsetup} sends the class of $\ell$ to the class represented by $\chi_{\ell}$, while $e_*$ is induced by the map in~\eqref{eq:definee}, itself induced by the map $e_{\omega_0}$ of~\eqref{map:definee}. To prove the lemma it is enough to show that, for every $\sigma\in G_K$, $e_{\omega_0}(\xi''_\ell(\sigma))$ and $\chi_\ell(\sigma)$ define the same element of $\mu_2(\Lbar)/\mu_2(\Kbar)$. 
			
			For any $\sigma \in G_K$, 
			\[
				g_\ell(\sigma)[\mm] 
				= g_\ell(\sigma)[2\omega_0] 
				= 2g_\ell(\sigma)[\omega_0] 
				= \sum_{\omega \in \Omega}\tilde\chi_\ell(\sigma)_\omega[\omega_0]\,.
			\]
			Since $[\omega] - [\omega_0] = \{\omega,\omega_0\}$, we may thus rewrite $\xi_\ell''(\sigma)$ as
			\[
				\xi''_\ell(\sigma) = \sum_{\omega \in \Omega}\tilde\chi_\ell(\sigma)_\omega\{\omega,\omega_0\}\\.
			\]
			
			Now let $\upsilon \in \Omega$. We may express $e_{\omega_0}(\xi''_\ell(\sigma))(\upsilon) =
			e_2(\xi''_\ell(\sigma), \{\upsilon,\omega_0\})$ as follows
			\[
			e_2(\xi''_\ell(\sigma), \{\upsilon,\omega_0\})
			= \prod_{\omega \in \Omega}e_2(\{\omega,\omega_0\},\{\upsilon,\omega_0\})^{\tilde\chi_\ell(\sigma)_\omega} 			= \prod_{\omega \ne \upsilon,\omega_0}e_2(\{\omega,\omega_0\},\{\upsilon,\omega_0\})^{\tilde\chi_\ell(\sigma)_\omega}.
			\]
			Observing that $e_2(\{\omega, \omega_0\},\{\upsilon,\omega_0\}) = - 1$ unless $\omega = \upsilon$, $\omega = \omega_0$ or $\upsilon = \omega_0$, it follows that
			\[
			e_{\omega_0}(\xi''_\ell(\sigma))(\upsilon) =
						e_2(\xi''_\ell(\sigma), \{\upsilon,\omega_0\})
						= \prod_{\omega \ne \upsilon,\omega_0}\chi_\ell(\sigma)_\omega, \quad \textup{ for any }v\neq \omega_0
			\]
			Finally, we note that $\prod_{\omega \in \Omega}\chi_\ell(\sigma)_\omega = 1$ as $\overline{\ell} \in \mathfrakL_1$ and obtain the desired conclusion, that $e_{\omega_0}(\xi''_\ell(\sigma))(\upsilon) = \chi_\ell(\sigma)_\upsilon\chi_\ell(\sigma)_{\omega_0}.$
		\end{proof}
		
		%%%%%%%%%%%%%%%%%%%%%%%%%%%%%%%%%%%%%%%%%%%%%%%%%%%%%%%%%%%%%%%%%%%%%%%%
		\subsection{The kernel of $(x-\alpha)$}%%%%%%%%%%%%%%%%%%%%%%%
		%%%%%%%%%%%%%%%%%%%%%%%%%%%%%%%%%%%%%%%%%%%%%%%%%%%%%%%%%%%%%%%%%%%%%%%%

		The kernel of $(x-\alpha)$ on the subgroup $\Pic^{(2)}C \subset \Pic C$ consisting of divisors classes of even degree is given by \cite[Theorem 11.3]{PS-descent}. Using this we derive the following.
		
		\begin{prop}
			\label{prop:dimImage}
			The kernel of $x-\alpha\,:\,\frac{\Pic C}{2 \Pic C} \to \frak{L}_c$ is generated by $\Pic C \cap 2\Pic_C(K)$ and the $K$-rational divisors on $C$ lying above $\infty \in \PP^1$. In particular, if $\Pic C = \Pic_C(K)$, then
			\[
				\ker\left(x-\alpha\right) = 
				\begin{cases}
					\langle [\mm] \rangle & \text{ if $c \notin K^{\times 2}$,}\\
					\langle [\mm], [\infty^+] \rangle & \text{ if $c \in K^{\times 2}$.}
				\end{cases}	
			\]
		\end{prop}
		
		\begin{proof}
			By~\cite[Proposition 11.1]{PS-descent} the kernel of 
			\[
				x-\alpha : \Pic^{(2)}C \to \frak{L}
			\]
			is generated by $\Pic C \cap 2\Pic_C(K)$ and $[\mm]$. Clearly $(x-\alpha)$ maps divisors of degree $m$ to classes with norm in $c^mK^{\times 2}$. In particular, if $c \notin K^{\times 2}$, then the kernel of $(x-\alpha)$ does not contain any divisor classes of odd degree. On the other hand, if $c \in K^{\times 2}$, then $[\infty^+]$ is defined over $K$ and lies in $\ker(x-\alpha)$. 
		\end{proof}
		
		%%%%%%%%%%%%%%%%%%%%%%%%%%%%%%%%%%%%%%%%%%%%%%%%%%%%%%%%%%%%%%%%%%%%%%%%%%%%
	\section{Proofs of the main theorems}\label{sec:Normc}%%%%%%%%%%%%%%%%%%%%%%
	%%%%%%%%%%%%%%%%%%%%%%%%%%%%%%%%%%%%%%%%%%%%%%%%%%%%%%%%%%%%%%%%%%%%%%%%%%%%

		For $n \ge 2$ define
		\begin{align*}
			\Br_nC &= \left\{ \calA \in \Br C \,:\, h(\calA) \in \textup{image}\left(\HH^1(K,J[n]) \to \HH^1(K,\Pic_C)\right) \right\}\,, \text{ and }\\
			\Br_2^\Upsilon C &= \left\{ \calA \in \Br C \,:\, h(\calA) \in \textup{image}\left(\ker(\Upsilon) \to \HH^1(K,\Pic_C)\right) \right\}\,,
		\end{align*}
		where $h : \Br C \to \HH^1(K,\Pic_C)$ is as in~\eqref{diag:defineh}, $\Upsilon:\HH^1(K,J[2]) \to \Br K[2]$ is as in~\eqref{eq:define_d} and the map $\HH^1(K,J[n]) \to \HH^1(K,\Pic_C)$ is induced by the inclusion $J[n] \subseteq J = \Pic^0_C \subseteq \Pic_C$.
		
		\begin{prop}
			\label{prop:L1exact}
			There is an exact sequence
			\[
				\Pic^0C \stackrel{x-\alpha}\To \mathfrakL_1\stackrel{\gamma}\To \Br_2^\Upsilon C/\Br_0 C \to 0\,.
			\]
		\end{prop}
		
		\begin{proof}
			From Proposition~\ref{prop:xi} and Lemma~\ref{lem:described'} it is clear that $h_0 \circ \gamma$ and $d$ give the same map $\mathfrakL_1 \to \HH^1(K,\Pic_C)$. Since~\eqref{eq:define_d} is exact, $\im(d) = \ker(\Upsilon)$, so $\gamma(\mathfrakL_1) = \Br_2^\Upsilon C/\Br_0 C$. The exactness stated in the proposition now follows immediately from Corollary~\ref{cor:descent}.
		\end{proof}
			
		\begin{lemma}
			\label{lem:Br_2}
				The index of $\Br_2C/\Br_0C$ in $(\Br C/\Br_0C)[2]$ divides $2$. If $\Pic^1_C(K) \ne \emptyset$ or $\Pic^1_C \notin 2\HH^1(K,J)$, then the index is $1$.
		\end{lemma}
		
		\begin{proof}
			
			Consider the following commutative diagram with exact rows.
			\[
				\xymatrix{
					0 \ar[r]&
					\left(\frac{\Br C}{\Br_0C}\right)[2]\ar[r]^{h_0}&
					\HH^1(K,\Pic_C)[2]\ar[r]&
					\HH^3(K,\G_m)\\
					0 \ar[r]&
					\frac{\Br_2C}{\Br_0C}\ar[r]\ar@{^{(}->}[u]&
					\frac{\HH^1(K,J)[2]}{\langle \Pic^1_C \rangle}\ar[r]\ar@{^{(}->}[u]&
					\HH^3(K,\G_m) \ar@{=}[u]\\				
				}
			\]
			The vertical map in the middle is an isomorphism if $\Pic^1_C(K)\ne\emptyset$ or $\Pic^1_C\notin 2\HH^1(K,J)$, and has its image has index $2$ otherwise.
		\end{proof}
		
		\begin{lemma}
			\label{lem:BrUpsilon}
			If $Br K[2] = 0$ or $\Omega$ admits a $G_K$-stable unordered partition into two sets of odd cardinality, then $\Br^\Upsilon_2C = \Br_2C$.
		\end{lemma}
		
		\begin{proof}
			Either assumption implies that $\Upsilon = 0$.
		\end{proof}
				
		\begin{remark}
			\label{rem:Br2neBrUps}
			In general one should not expect that $\Br_2^\Upsilon C = \Br_2C$. For example, if $K$ is a $p$-adic field, $\Pic^1_C(K) \ne \emptyset$ and $\Omega$ does not admit a $G_K$-stable partition into two sets of odd cardinality, then $\Br_2^\Upsilon C \ne \Br_2C$. To see this, recall that the cup product on $\HH^1(K,J[2])$ is nondegenerate (see \cite[\S2]{Tate-duality}). The above assumptions therefore imply that there exists some $T \in \HH^1(K,J[2])$ such that $\Upsilon(T) \ne 0$. Let $T' \in \HH^1(K,\Pic_C)$ be the image of $T$. Then every lift of $T'$ to $\HH^1(K,J[2])$ is of the form $\tilde T = T + \delta(P)$ for some $P \in J(K)$, and none of them lie in $\ker(\Upsilon)$ since $\Upsilon(\tilde{T}) = \Upsilon(T) + \Upsilon(\delta(P)) = \Upsilon(T) \ne 0$. Here $\Upsilon(\delta(P)) = 0$ since $\Pic^1_C$ lies in the image of $\delta$, which is self-orthogonal with respect to the pairing (ibid.). This shows that if $\calA \in \Br_2C$ is such that $h(\calA) = T'$, then $\calA \notin \Br_2^\Upsilon C$. Moreover, such an $\calA$ exists as $\HH^3(K,\G_m) = 0$.
		\end{remark}
		
		%%%%%%%%%%%%%%%%%%%%%%%%%%%%%%%%%%%%%%%%%%%%%%%%%%%%%%%%%%%%%%%%%%%%%%%%
		\subsection{Proof of Theorems~\ref{thm:EvenOddThm} and \ref{thm:OddHypThm}}
		%%%%%%%%%%%%%%%%%%%%%%%%%%%%%%%%%%%%%%%%%%%%%%%%%%%%%%%%%%%%%%%%%%%%%%%%
			In the odd case we have already seen that $\gamma$ maps $\mathfrakL_1$ to $\Br^0C$ (Lemma~\ref{lem:Br^0}). Using  Lemmas~\ref{lem:Br_2} and \ref{lem:BrUpsilon} we see that the hypotheses imply that $ \Br_2^\Upsilon C/\Br_0 C = (\Br C/\Br_0C)[2]$, so the theorems follow from Proposition~\ref{prop:L1exact}.\qed
	
		%%%%%%%%%%%%%%%%%%%%%%%%%%%%%%%%%%%%%%%%%%%%%%%%%%%%%%%%%%%%%%%%%%%%%%%%
		\subsection{Proof of Theorems~\ref{thm:ComputeH1} and~\ref{thm:ExactIfC1}}%%%%%%%%%%%%%%%%%%%%%%
		%%%%%%%%%%%%%%%%%%%%%%%%%%%%%%%%%%%%%%%%%%%%%%%%%%%%%%%%%%%%%%%%%%%%%%%%

		\begin{lemma}
			\label{lem:Lccoset}
			We have that $\gamma(\mathfrakL_c\setminus\mathfrakL_1) \not\subseteq \Br_2C/\Br_0C$ if and only if $\Pic^1_C(K) = \emptyset$ and $\mathfrakL_c\neq\mathfrakL_1$.
		\end{lemma}
	
		\begin{proof}
			The statement is trivially true when $\mathfrakL_c = \mathfrakL_1$. So suppose $\ell \in L^\times$ is a representative for a class $\overline{\ell} \in \mathfrakL_c \setminus \mathfrakL_1$. Then $h_0\circ\gamma(\overline{\ell})$ is represented by the cocycle $\xi_\ell \in C^1(K,\Pic_C)$ of Proposition~\ref{prop:xi}. 
			Using that $2[\omega] = [\mm]$ in $\Pic\Cbar$ we have
			\begin{align*}
				2\xi_\ell(\sigma) &= \left(\sum_{\omega \in \Omega} \tilde\chi_\ell(\sigma)_\omega2[\omega]\right) - 2g_\ell(\sigma)[\mm] - 2\tilde{\chi}_c(\sigma)[\infty^+]\\
				&= \tilde\chi_c(\sigma) [\mm] - \tilde\chi_c(\sigma)2[\infty^+]\\
				&= \tilde\chi_c(\sigma) ([\infty]^- - [\infty^+])\,.
			\end{align*}
			This shows that, when considered as a cocycle taking values in $\Pic^0\Cbar = J(\Kbar)$, $2\xi_\ell$ represents the class of $\Pic^1_C$ in $\HH^1(K,J)$. This class is trivial if and only if $\Pic^1_C(K) \ne \emptyset$. The lemma now follows easily from the definition of $\Br_2C$.
		\end{proof}
		
		\begin{proof}[Proof of Theorem~\ref{thm:ComputeH1}]
			We must show that the complex
			\[
				\Pic(C) \stackrel{x-\alpha}\To \mathfrakL_c \stackrel{\gamma}\To \left(\frac{\Br C}{\Br_0C}\right)[2]
			\]
			is exact except possibly if $\Pic^1C = \emptyset \ne \Pic^1_C(K)$ and $\bar c \in \Norm_{L/K}(\mathfrakL)$, in which case $\ker(\gamma)/\im(x-\alpha)$ is generated by any element of $\mathfrak{L}_c\setminus\mathfrak{L}_1$.

			Consider the following commutative diagram. 
			\begin{equation}
				\label{diag:Lcextension}
				\xymatrix{
					\frac{\Pic^0C}{2\Pic^0C} \ar[r]^{(x-\alpha)} \ar[d] &
					\mathfrakL_1 \ar[r]^\gamma \ar[d] &
					\frac{\Br_2 C}{\Br_0 C}\ar[d]\\
					\frac{\Pic C}{2\Pic C} \ar[r]^{(x-\alpha)} &
					\mathfrakL_c \ar[r]^\gamma &
					\frac{\Br C}{\Br_0 C}
				}
			\end{equation}
			The top row is exact by Proposition~\ref{prop:L1exact}, and the bottom row is a complex by Theorem~\ref{thm:MainComplex}.
			
			Let us first consider the case when $\mathfrakL_c = \mathfrakL_1$. This happens if and only if $c \in K^{\times 2}$ or $\bar{c} \notin \Norm_{L/K}(\mathfrakL)$. When $c \in K^{\times 2}$ we have $[\infty^+] \in \ker(x-\alpha)$, and when  $\bar{c} \notin \Norm_{L/K}(\mathfrakL)$ there are no $K$-rational divisor classes of odd degree \cite[Corollary 4.4]{CreutzANTSX}. Both possibilities imply that $(x-\alpha)(\Pic C) = (x-\alpha)(\Pic^0C)$, and so exactness of the bottom row follows from exactness of the top row of~\eqref{diag:Lcextension}.
			
			Now we consider the case $\mathfrakL_c \ne \mathfrakL_1$, which implies that $\bar c \in \Norm_{L/K}(\mathfrakL)$. Then $(x-\alpha)$ sends $K$-rational divisor classes of odd degree to $\mathfrakL_c \setminus \mathfrakL_1$ \cite[Lemma 4.3]{CreutzANTSX}. If $[\mathfrakL_c:\mathfrakL_1] = \left[ \frac{\Pic C}{2\Pic C} : \frac{\Pic^0C}{2\Pic^0C} \right] = 2$, then  exactness follows from the fact that the top row is exact and the bottom row is a complex. So we may assume there are no $K$-rational divisors of odd degree. Then $(x-\alpha)(\Pic C) = (x-\alpha)(\Pic^0C) \subseteq \mathfrakL_1$, and exactness follows from exactness of the top row of~\eqref{diag:Lcextension}, except possibly if $\gamma(\mathfrakL_c) \cap \gamma(\mathfrakL_1) \ne \emptyset$. Lemma~\ref{lem:Lccoset} shows that this can only happen when $\Pic^1_C(K) \ne \emptyset$.
		\end{proof}
		
		\begin{proof}[Proof of Theorem~\ref{thm:ExactIfC1}]
		We must show that $\gamma(\mathfrakL_c) = (\Br C/\Br_0C)[2]$ when $\Br K[2] = 0$. By Proposition~\ref{prop:L1exact} and Lemma~\ref{lem:BrUpsilon} the assumption on $\Br K[2]$ implies that $\gamma(\mathfrakL_1) = \Br_2C/\Br_0C$. Lemma~\ref{lem:Br_2} allows us to further assume that $\Pic^1_C$ is nonzero and divisible by $2$ in $\HH^1(K,J)$. The index of $\Br_2C/\Br_0C$ in $(\Br C/\Br_0C)[2]$ divides $2$, so using Lemma~\ref{lem:Lccoset} it suffices to show that $\mathfrakL_c\setminus\mathfrakL_1 \ne \emptyset$. We know that $c \notin K^{\times 2}$, otherwise $\Pic^1_C$ would be trivial in $\HH^1(K,J)$. So we are reduced to showing that there exists some $\ell \in L^\times$ such that $\Norm_{L/K}(\ell) \in cK^{\times 2}$. 
		
		For this we will make use of the theory of torsors under groups of multiplicative type as described in \cite[Part I]{Skorobogatov-torsors}. For $X = C$ or $X = \Pic^1_C$, let $\lambda_n$ denote the canonical embedding $\lambda_n: J[n] \cong \Pic^0_X[n] \subseteq \Pic_X$. An \defi{$n$-covering  of $X$} is an $X$-torsor under $J[n]$ of type $\lambda_n$. Since $\Pic^1_C \in 2 \HH^1(K,J)$, there exists a $2$-covering $T \to \Pic^1_C$ (see \cite[Proposition 3.3.5]{Skorobogatov-torsors}). Pulling this back along the canonical embedding $C \to \Pic^1_C$ gives a $2$-covering $\psi:Y \to C$. For any $\omega \in \Omega$ the pull back $\psi^*[\omega]$ is a $K$-rational divisor class on $Y$. 
		
		If $\psi^*[\omega] \in \Pic Y$, then it induces a projective embedding of $Y$ in which the pull backs of the ramification points on $C$ are hyperplane sections. Up to composition with the hyperelliptic involution on $C$, the $2$-coverings of $C$ with a model of this type are parameterized by the elements in the set $\left\{ \overline{\ell} \in\mathfrakL\,:\, \Norm_{L/K}(\ell) \in cK^{\times 2} \right\}$ (see \cite[\S3]{BruinStoll} or \cite[Proposition 5.4]{CreutzANTSX}). In particular, it will suffice to show that $\psi^*[\omega] \in \Pic Y$, for then there exists some $\ell \in L^\times$ with norm in $cK^{\times2}$.
		
		The obstruction to a rational divisor class being represented by a rational divisor is given by a well known exact sequence, $0 \to \Pic Y \to \Pic_Y(K) \stackrel{\theta}\to \Br K\,$. In our situation, $2\psi^*[\omega] = \psi^*[2\omega] = \psi^*[\mm] \in \Pic Y$. So $\theta(\psi^*[\omega]) \in \Br K[2]$, which is trivial by assumption. This completes the proof.
		\end{proof}
		
		\begin{remark}
			If one is willing to assume that $K$ is $C_1$, then the final argument of the proof above can be simplified: the equation $\Norm_{L/K}(\ell) = c a^{[L:K]}$ with $\ell \in L$ and $a \in K$ gives a homogeneous equation of degree $[L:K]$ in $[L:K] + 1$ variables. If $K$ is $C_1$, then it must have a solution. 
		\end{remark}

	%%%%%%%%%%%%%%%%%%%%%%%%%%%%%%%%%%%%%%%%%%%%%%%%%%%%%%%%%%%%%%%%%%%%%%%%%%%%
	\section{Relation to the Cassels-Tate pairing}\label{sec:CTpairing}%%%%%%%%%
	%%%%%%%%%%%%%%%%%%%%%%%%%%%%%%%%%%%%%%%%%%%%%%%%%%%%%%%%%%%%%%%%%%%%%%%%%%%%
		Throughout this section $K$ is a number field. Let $X$ be a smooth, projective, and geometrically integral variety $X$ over $K$. There is a well known pairing due to Manin,
		\[	
			\langle\cdot,\cdot\rangle_{\Br} : \Br(X) \times X(\A_K) \To \Q/\Z\,, \quad \langle \calA,(P_v)\rangle_{\Br} \mapsto \sum_v \inv_v\eval_{P_v}(\calA)\,, 
		\]
		where the sum runs over all places of $K$. By the Hasse reciprocity law, the left kernel contains $\Br_0(X)$ and the right kernel contains the diagonal image of $X(K)$ in $X(\A_K)$. For any subgroup $B \subseteq \Br(X)$, we denote by $X(\A_K)^B$ the subset of $X(\A_K)$ which is orthogonal to $B$ with respect to the pairing. Define 
		\begin{align*}
			\Br_\sha X &= \left\{ \calA \in \Br X : h(\calA) \in \im\left(\Sha(\Pic^0_X) \to \HH^1(K,\Pic_X)\right)\right\}\,,
		\end{align*}
		where for an abelian variety $A$ over $K$, $\Sha(A)$ denotes its Tate-Shafarevich group.
		
		The following is a slight generalization of \cite[Theorem 6.2.3]{Skorobogatov-torsors}, which is due to Manin. Similar methods have been used to give a conditional proof that the Brauer-Manin obstruction to $0$-cycles of degree $1$ on smooth projective curves is the only one (see \cite[Theorem 1.1]{ES-0cycles}, \cite[Proposition 3.7]{CT-0cycles}, \cite[Theorem 8.4]{Saito-0cycles}). As a corollary we observe that \cite[Corollary 7.7]{Stoll-descentob} holds for all curves, not just those possessing a $K$-rational divisor class of degree $1$.
		
		\begin{theorem}
			\label{thm:Manin}
			Assume that $X(\A_K) \ne \emptyset$. Let $A = \Alb^0_X$, $V = \Alb^1_X$ and suppose $\calA\in \Br_\sha X$ is such that $h(\calA)$ is the image of $W \in \Sha(\Pic^0_X) = \Sha(A^\vee)$. Then, for any adelic point $(P_v) \in X(\A_K)$, 
			\[ 
				\langle \calA ,(P_v) \rangle_{\Br} = - \langle V, W \rangle_\textsc{ct}\,, 
			\] 
			where $\langle\cdot,\cdot\rangle_\textsc{ct}$ denotes the Cassels-Tate pairing on $\Sha(A)\times \Sha(A^\vee)$. In particular, $X(\A_K)^\calA$ is either empty or equal to $X(\A_K)$, and $X(\A_K)^{\Br_{\ssha} X}= \emptyset$ if and only if $\Alb^1_X$ is not divisible in $\Sha(A)$.
		\end{theorem}		
		
		\begin{cor}
			\label{cor:descob=Brob}
			If $X$ is a smooth, projective, and geometrically integral curve, then for any $n$, 
			\[
				X(\A_K)^{n\textup{-ab}} = X(\A_K)^{\Br X[n]}\,,
			\]
			i.e., the adelic information coming from an $n$-descent is precisely the information coming from the $n$-torsion in the Brauer group.
		\end{cor}
		
		\begin{remark}\label{rem:noextrainfo}
			The set $X(\A_K)^{n\textup{-ab}}$ is defined in \cite{Stoll-descentob}; \cite[Theorem 6.1.2]{Skorobogatov-torsors} shows that $X(\A_K)^{n\textup{-ab}} = X(\A_K)^{\Br_nX}$, where $\Br_nX \subseteq \Br X$ is as defined at the beginning of \S\ref{sec:Normc}. Thus the corollary can also be interpreted as saying that the elements of $\Br X[n] \setminus \Br_nX$ provide no additional information regarding the adelic points of $X$. In fact, the proof below shows that the elements of $\Br X[n]\setminus \Br_nX$ provide no information whatsoever.
		\end{remark}
	
		\begin{proof}[Proof of Corollary~\ref{cor:descob=Brob}]
			As remarked above, $X(\A_K)^{\Br X[n]} \subseteq X(\A_K)^{\Br_nX} = X(\A_K)^{n\textup{-ab}}$. So if $X$ has no locally solvable $n$-coverings, then both sets in question are empty. We may thus assume that $X$ has an everywhere locally solvable $n$-covering. This implies that $\Pic^1_X = nW$ for some $W \in \Sha(J)$ and that $\langle\cdot,\cdot\rangle_\textsc{ct}$ is alternating \cite{PoonenStollCT}.
			Now suppose $w \in \Br_\sha X$ has the same image in $\HH^1(K,\Pic_X)$ as $W$. For any adelic point $(P_v) \in X(\A_K)$, applying the theorem gives:
			\[
				\langle w, (P_v)\rangle_{\Br} = \langle \Pic^1_X,W\rangle_\textsc{ct} = \langle nW,W\rangle_\textsc{ct} = n\langle W,W\rangle_\textsc{ct}\,,
			\]
			which is trivial since the pairing is alternating. Hence, $X(\A_K)^w = X(\A_K)$.
			
			In the exact sequence,
			\[
				\Z \to \HH^1(K,\Pic^0_X) \to \HH^1(K,\Pic_X) \to 0\,,
			\]
			$1 \in \Z$ maps to the class of $\Pic^1_X$. It follows that the quotient of $(\Br X/\Br_0X)[n]$ by $\Br_{n}X/\Br_0X$ is cyclic and generated by the image of $w$. The result follows since we have shown that $w$ does not obstruct any adelic points.
		\end{proof}
			
		\begin{proof}[Proof of Theorem~\ref{thm:Manin}]
			For the case that $X$ is a torsor under an abelian variety ({e.g.}, $X = V$) see \cite[6. Th\'eor\`eme]{Manin-BMobs} or \cite[Theorem 6.2.3]{Skorobogatov-torsors}. To derive the general result from this, note that the canonical morphism $\phi:X\to V$ induces an isomorphism $\Pic^0_X \cong \Pic^0_V$, and consequently a commutative diagram,
			\[
				\xymatrix{
					\Sha(\Pic^0_V) \ar@{=}[d]^{\phi^*}\ar@{^{(}->}[r]
					&\HH^1(K,\Pic^0_V) \ar@{=}[d]^{\phi^*} \ar[r]
					&\HH^1(K,\Pic_V) \ar[d]^{\phi^*}
					&\Br_1V/\Br_0V \ar@{=}[l] \ar[d]^{\phi^*}\\
					\Sha(\Pic^0_X) \ar@{^{(}->}[r]
					&\HH^1(K,\Pic^0_X) \ar[r]
					&\HH^1(K,\Pic_X)
					&\Br_1X/\Br_0X \ar@{=}[l]\,.
				}
			\]
			Suppose $W \in \Sha(\Pic^0_X)$ and $\calA \in \Br_\sha X$ are as in the statement. From the diagram above it is clear that there exists $\calA' \in \Br_\sha V$ such that $\phi^*\calA' \equiv \calA \bmod \Br_0X$. Then we have
			\[
				\langle \calA,(P_v)\rangle_{\Br} = \langle \phi^*\calA',(P_v)\rangle_{\Br} = \langle \calA', \phi(P_v)\rangle_{\Br} = -\langle V, W\rangle_\textsc{ct}\,,
			\]
			since the theorem holds for $V$.
			
			The final statement follows from the fact that the left and right kernels of the Cassels-Tate pairing are the maximal divisible subgroups~\cite{Tate-duality}.			
		\end{proof}

		%%%%%%%%%%%%%%%%%%%%%%%%%%%%%%%%%%%%%%%%%%%%%%%%%%%%%%%%%%%%%%%%%%%%%%%%
		\subsection{Computing Brauer-Manin Obstructions}%%%%%%%%%%%%%%%%%%%%%%%%
		%%%%%%%%%%%%%%%%%%%%%%%%%%%%%%%%%%%%%%%%%%%%%%%%%%%%%%%%%%%%%%%%%%%%%%%%
	
			The following proposition proves Theorem~\ref{thm:LocallyConstantsInImage}.
			\begin{prop}\label{prop:BrShainBrUps}
				Let $C$ be a double cover of $\PP^1_K$ with $C(\A_K) \ne \emptyset$. Then 
				\[
					(\Br_\sha C/\Br_0C)[2] \subseteq \gamma(\mathfrakL_c)\,.
				\]
			\end{prop}
			
			\begin{proof}
				Set $\Br_{\sha,2}C = (\Br_\sha C) \cap (\Br_2C)$. By \cite[Theorem 13.3]{PS-descent}, the subgroup of $\HH^1(K,J[2])/\langle \Pic^1_C \rangle$ mapping into $\Sha(J)[2]/\langle \Pic^1_C \rangle$ is contained in the kernel of $\Upsilon$. It follows that $\Br_{\sha,2}C \subseteq \Br^\Upsilon_2C$, and so $\Br_{\sha,2}C/\Br_0C \subseteq \gamma(\mathfrakL_1)$ by Proposition~\ref{prop:L1exact}. If $\Br_\sha C[2] \subseteq \Br_{\sha,2}C$, then there is nothing more to prove. Hence we may assume that there exists some $w \in \Br_\sha C[2] \setminus \Br_{\sha,2}C$. Then, as in the proof of Corollary~\ref{cor:descob=Brob}, the quotient of $(\Br_{\sha}C/\Br_0C)[2]$ by $\Br_{\sha,2}C/\Br_0C$ is of order $2$.
				
				The existence of $w$ implies that there exists $W \in \Sha(J)$ such that $2W = \Pic^1_C \ne 0$. By \cite[Theorem 4.6]{CreutzANTSX} this implies that there exists some $\overline\ell \in \mathfrakL_c\setminus \mathfrakL_1$ such that $\res_v(\overline\ell) \in (x-\alpha)(\Pic^1C_{K_v})$, for every completion $K_v$ of $K$. Since $\gamma\circ(x-\alpha) = 0$ by Theorem~\ref{thm:MainComplex}, we must have $\gamma(\overline\ell) \in \Br_{\sha}C/\Br_0C$. On the other hand, $\gamma(\overline\ell) \notin \Br_2C/\Br_0C$ by Lemma~\ref{lem:Lccoset}. Thus $\gamma(\overline\ell)$ must generate the quotient of $(\Br_{\sha}C/\Br_0C)[2]$ by $\Br_{\sha,2}C/\Br_0C$. Therefore, $(\Br_{\sha}C/\Br_0C)[2] \subseteq \gamma(\mathfrakL_c)$.
			\end{proof}
								
			\begin{remark}\label{rem:Algorithm}
				Regardless of whether $C$ is locally solvable or not, the proof of Corollary~\ref{cor:descob=Brob} shows that $C(\A_K)^{(\Br_{\ssha}C)[2]} = C(\A_K)^{\Br_{\ssha,2}C}$. When $C$ has a $K_v$-rational divisor of degree $1$ for every completion $K_v$ of $K$, then $\Br_{\ssha,2}C/\Br_0C \subseteq \gamma(\mathfrakL_1)$. In this case the subgroup of $\mathfrakL_1$ mapping into $\Br_\sha C/\Br_0C$ is the \defi{fake $2$-Selmer group of $J$}, denoted $\Sel_\textup{fake}^2(J)$. An algorithm for computing it is described in \cite{PS-descent}. Together with the following proposition, this gives a practical algorithm for computing the induced map
				\[
					\Sel_\textup{fake}^2(J) \to \frac{\Sha(J)[2]}{\langle \Pic^1_C \rangle} \stackrel{\langle \Pic^1_C,\cdot\rangle}\To \Q/\Z\,,				
				\]
				at least when $C(\A_K) \ne \emptyset$.
			\end{remark}

			\begin{prop}\label{prop:application}
				Suppose $C : y^2 = cf(x)$ is an even double cover of $\PP^1$ defined over $K$ with $C(\A_K) \ne \emptyset$ and that the coefficients of $cf(x)$ are integral. Let $\beta \in \Sha(J)$, and suppose $\ell$ represents $\overline\ell \in \mathfrakL_c$ such that $d(\overline{\ell})$ and $\beta$ give the same class in $\Sha(J)/\langle \Pic^1_C\rangle$. Then, for any $(P_v) \in C(\A_K)$. 
					\[
						\langle \Pic^1_C,\beta\rangle_\textsc{ct} = 
						\sum_{v \in S} \inv_v \eval_{P_v} \Cor_{\kk(C_L)/\kk(C)}(\ell, (x-\alpha))_2 \,,
					\]
					The sum here runs over the primes in the finite set $S$ consisting of all primes of $K$ appearing with multiplicity greater or equal to $2$ in $4c^2\cdot\textup{disc}(f)$ and all archimedean primes.
			\end{prop}
			
			\begin{proof}
				If $v$ does not lie in $S$, then both $(x - \alpha)([P_v])$ and $\ell$ have even valuation at all primes $w$ above $v$, by \cite[Lemma 4.3]{BruinStoll} and \cite[Proposition 5.10]{Stoll-2descent}. For such $v$ the invariant $\inv_v \eval_{P_v} \Cor_{\kk(C_L)/\kk(C)}(\ell, (x-\alpha))_2 = 0$.
			\end{proof}
			
		%%%%%%%%%%%%%%%%%%%%%%%%%%%%%%%%%%%%%%%%%%%%%%%%%%%%%%%%%%%%%%%%%%%%%%%%
		\subsection{An Example}%%%%%%%%%%%%%%%%%%%%%%%%%%%%%%%%%%%%%%%%%%%%%%%%%
		%%%%%%%%%%%%%%%%%%%%%%%%%%%%%%%%%%%%%%%%%%%%%%%%%%%%%%%%%%%%%%%%%%%%%%%%
		
		\begin{theorem}\label{thm:Example}
			Let $c$ be a square free integer, let $C$ be the locally solvable double cover of $\PP^1_\Q$ given by 
			\[
				C : y^2 = c(x^2+1)(x^2+17)(x^2-17)\,.
			\]
			Then $(-1,x^2-17)_2 \in  \Br_\sha C$, and if $W\in\Sha(J)$ denotes a corresponding torsor, then 
			\[
				\langle \Pic^1_C,W \rangle_\textsc{ct} =
				\frac{\#\left\{ p \mid c : \text{$p$ is an odd prime, and }\left(\frac{17}{p}\right) = \left(\frac{-1}{p}\right) = -1\right\}}{2} + \frac{\textup{sign}(c) - 1}{4}
			\]
			Furthermore, if $\langle \Pic^1_C,W \rangle_\textsc{ct} = 1/2$, then $\dim_{\F_2}\Sha(J)[2] \ge 2$ and neither $W$ nor $\Pic^1_C$ is divisible by $2$ in $\HH^1(\Q,J)$.
		\end{theorem}
			
		\begin{proof}
			We first note that $(-1,x^2-17)_2 = \gamma'(\ell)$, for the element 
			\[
				\ell = (1,1,-1) \in \Q(\sqrt{-1})\times \Q(\sqrt{-17})\times\Q(\sqrt{17}) \simeq L.
			\]
			It is easy to see that $C$ is locally solvable. In fact, it has a $\Q_p$-rational ramification point for every prime $p$. One can also check that $\res_p(\overline\ell) \in (L\otimes\Q_p)^\times/\Q_p^\times(L\otimes\Q_p)^{\times 2}$ is trivial for every prime $p$ (this is weaker than requiring $\res_p(\ell) \in \Q_p^\times (L\otimes\Q_p)^{\times 2}$ everywhere locally). This imples that $\gamma'(\ell) \in \Br_\sha C$. Consequently there is a torsor $W \in \Sha(J)$ whose class in $\Sha(J)/\langle\Pic^1_C\rangle$ is represented by $d(\overline\ell)$. By Theorem~\ref{thm:Manin} and Proposition~\ref{prop:application}, for any $(P_p) \in C(\A_\Q)$, we have
			\begin{equation}
				\langle \Pic^1_C, W \rangle_\textsc{ct} 
				= \langle \gamma'(\ell), (P_p) \rangle_{\Br}
				= \sum_p \inv_p \left(\eval_{P_p}(-1,x^2-17)_2\right)\,,
			\end{equation}
			To ease notation, let us set $\varepsilon_p = \inv_p \left(\eval_{P_p}(-1,x^2-17)_2\right)$. Note that, by Theorem~\ref{thm:Manin}, $\varepsilon_p$ depends on $c$, but not on the subsequent choice for $P_p$. 	
			\begin{lemma}\label{lem:epsp} Let $p$ be an odd prime. Then 
				\begin{align*}
					\varepsilon_p &= 				
						\begin{cases}
							1/2 & \textup{if } \left(\frac{-1}{p}\right) =  \left(\frac{-17}{p}\right) = -1\text{ and } p \mid c\,,\\
							0 & \textup{if else}\,.
						\end{cases}\\
					\varepsilon_2 &= 
						\begin{cases}
							0 & \textup{if }c \equiv 1,2\text{ or }5 \mod 8\,,\\
							1/2 & \textup{if }c \equiv 3,6\text{ or } 7\mod 8\,;
						\end{cases}\\
					\varepsilon_\infty &= 
						\begin{cases}
							0 & \textup{if }c > 0\,,\\
							1/2 & \textup{if }c <0\,;
						\end{cases}
				\end{align*}
			\end{lemma}
			The lemma is proved below; using it gives:
			\begin{align*}
				\varepsilon_2 &= \#\left\{ p \mid c: { \left(\frac{-1}{p}\right)} = -1 \right\}/2\\
				&=\#\left\{ p \mid c : \left(\frac{-1}{p}\right) =  \left(\frac{17}{p}\right) = -1\right\}/2+ \#\left\{ p \mid c : \left(\frac{-1}{p}\right) =  \left(\frac{-17}{p}\right) = -1\right\}/2 \\
				&= \#\left\{ p \mid c : \left(\frac{-1}{p}\right) =  \left(\frac{17}{p}\right) = -1\right\}/2+ \sum_{p \mid c} \varepsilon_p\,,
			\end{align*}
			from which the formula in the theorem follows easily.
			
			Now let us prove the final statement of the theorem. Since $C(\A_K) \ne \emptyset$, {the pairing $\langle\cdot,\cdot\rangle_\textsc{ct}$ is alternating \cite{PoonenStollCT}. Tate's proof that the left and right kernels of the pairing are the maximal divisible subgroups \cite[Theorem 3.2]{Tate-duality} shows that $\langle\cdot,\cdot\rangle_\textsc{ct}$ induces a nondegenerate alternating pairing on $\Sha(J)[2]/2\Sha(J)[4]$. As is well known, this implies that the order of this group is a square. If $\langle\Pic^1_C,W\rangle_\textsc{ct} = 1/2$, then the group is nontrivial, and hence has dimension at least $2$. }To show that this also implies that $\Pic^1_C \notin 2\HH^1(K,J)$, we use \cite[Theorem 3]{CreutzShaDiv}, which states that an element of $\Sha(J)$ is divisible by $n$ in $\HH^1(K,J)$ if and only if it pairs trivially with the image of $\Sha^1(K,J[n])$ in $\Sha(J)[n]$. In our situation we know that $W$ lies in this image of $\Sha^1(K,J[2]) \to \Sha(J)$, because $\overline\ell$ is locally trivial.
			\end{proof}		
			
			\begin{proof}[Proof of Lemma~\ref{lem:epsp}]
				Suppose $p$ is odd.  If $\left(\frac{17}{p}\right) = -1$, then $\ell$ is trivial since $-1$ is a square in $\Q(\sqrt{17})\otimes \Q_p$. So suppose $\left(\frac{17}{p}\right) = 1$, let $a \in \Q_p$ be a square root of $17$ and set $P_p = (a,0) \in C(\Q_p)$. Then
				\begin{align*}
					\varepsilon_p
					&= \inv_p\eval_{(a,0)}(-1,x^2-17)_2\\
					&= \inv_p\eval_{(a,0)}\left(-1,c(x^2+17)(x^2+1)\right)_2\\
					&= \inv_p(-1,c\cdot2^2\cdot3^2\cdot17)_2\\
					&= \inv_p(-1,c)_2\,,
				\end{align*}
				which is nontrivial if and only if $p \mid c$ and $\left(\frac{-1}{p}\right) = -1$. To arrive at the statement in the lemma, note that if $\left(\frac{-1}{p}\right) = -1$ and $\left(\frac{17}{p}\right) = 1$, then $\left(\frac{-17}{p}\right) = -1$.

				Clearly $\varepsilon_2$ depends only on the class of $c$ in $\Q_2^\times/\Q_2^{\times2}$. The table below gives, for each square class, a value $x(P_2) \in \Z$ for which $f(x(P_2)) \equiv c \bmod \Q_2^{\times 2}$, i.e., $x(P_2)$ is the $x$-coordinate of a $\Q_2$-point on the curve $y^2 = cf(x)$. The corresponding invariant is then
				\[
					\varepsilon_2 = \inv_2 (-1, x(P_2)^2-17)_2\,,
				\]
				The claim above follows immediately from the table.
				\[
				\begin{tabular}{|c|cccccccc|}
						\hline
						$c \bmod \Q_2^{\times2}$ & $1$ & $2$ & $3$ & $5$ & $6$ & $7$ & $10$ & $14$ \\ \hline
						$x(P_2)$ & $9$ & $5$ & $2$ & $15$ & $13$ & $0$ & $11$ & $3$\\ \hline
						$\varepsilon_2$ & $0$ & $0$ &$1/2$ & $0$ & $1/2$ & $1/2$ & $0$ & $1/2$\\\hline
					\end{tabular}
				\]
							
				For any real point $P_\infty \in C(\R)\setminus \Omega$, $\varepsilon_\infty = \inv_\infty (-1,x(P_\infty)^2-17)_2$,
				which can be nonzero if and only if there are real points with $|x(P_\infty)| < \sqrt{17}$, which occurs if and only if $c < 0$.	
			\end{proof}

%%%%%%%%%%%%%%%%%%%%%%%%%%%%%%%%%%%%%%%%%%%%%%%%%%%%%%%%%%%%%%%%%%%%%%%%%%%%%%%%
%%%%%%%%%%%%%%%%%%				Bibliography			%%%%%%%%%%%%%%%%%%%%%%%%
%%%%%%%%%%%%%%%%%%%%%%%%%%%%%%%%%%%%%%%%%%%%%%%%%%%%%%%%%%%%%%%%%%%%%%%%%%%%%%%%

\end{document}